\documentclass[10pt]{article}
\usepackage[a4paper]{geometry}
\usepackage{amsmath, amssymb, amsthm}

\usepackage{appendix}
\usepackage{multicol} 
\usepackage{makeidx}
\usepackage{cite}
\usepackage{captdef}
\usepackage{paralist}
\usepackage{ dsfont }

\usepackage{mathrsfs}

\usepackage[colorlinks=true]{hyperref}
\hypersetup{urlcolor=blue, citecolor=red}

\usepackage{psfrag}
\usepackage{graphicx}
\usepackage{graphics}
\usepackage[dvipsnames]{xcolor}
\usepackage{comment}
\usepackage[showonlyrefs]{mathtools}
\mathtoolsset{showonlyrefs=true}

\usepackage[cm]{fullpage}

\usepackage{subcaption}

    \newtheorem{theo}{Theorem}[section]
    \newtheorem{coro}[theo]{Corollary}
    \newtheorem{lemma}[theo]{Lemma}
    \newtheorem{prop}[theo]{Proposition}
    \theoremstyle{definition}
    \newtheorem{defi}[theo]{Definition}
    \newtheorem{remark}[theo]{Remark}

\newcommand{\R}{\mathbb{R}}

\renewcommand{\dim}{{\rm dim}\,}
\newcommand{\eps}{\varepsilon}

\newcommand{\abs}[1]{\left|#1\right|}
\newcommand{\norm}[1]{\left \| #1\right \|}

\DeclareMathOperator{\sign}{sign}

\renewcommand{\epsilon}{\varepsilon}

\newcommand{\sgn}{\operatorname{sign}}

\numberwithin{equation}{section}

\def\sideremark#1{\ifvmode\leavevmode\fi\vadjust{\vbox to0pt{\vss
 \hbox to 0pt{\hskip\hsize\hskip1em
 \vbox{\hsize2.1cm\tiny\raggedright\pretolerance10000
  \noindent #1\hfill}\hss}\vbox to15pt{\vfil}\vss}}}%

\title{On Neumann $p$-Laplacian Lane-Emden equations and their asymptotic relationship with relative isoperimetric problems}
\author{Sean McCurdy, Alberto Saldaña, \& Delia Schiera}

\date{}

\begin{document}
\maketitle

\begin{abstract}

We consider a family of pure Neumann $p$-Laplacian problems, including eigenvalue problems, Lane--Emden type equations, and extremal cases such as sign nonlinearities and the $1$-Laplacian. Using variational methods, we develop a unified framework that establishes existence of solutions and characterizes their asymptotic behavior as the parameters vary. This approach reveals a natural asymptotic connection between pure Neumann $p$-Laplacian equations and a relative isoperimetric problem known as the Neumann-Cheeger problem. We describe the shape of minimizers in domains with different geometries and obtain results on regularity, uniqueness, multiplicity, symmetry, and symmetry breaking phenomena.

\medskip

\noindent \textbf{Mathematics Subject Classification:} 
(primary) 
35P30, 
35J20, 
35J25, 
35J70, 
49Q10, 
(secondary) 
35B40, 
49J52, 
35B06. 

\medskip

\end{abstract}

\section{Introduction}

Consider the following pure Neumann $p$-Laplacian nonlinear eigenvalue problem
\begin{align}\label{pLap}
    -\Delta_p u = \lambda_{p,q}\,  f_q(u)\quad \text{ in }\Omega,\qquad \partial_\nu u = 0\quad \text{ on }\partial \Omega, 
\end{align}
where $\Omega\subset \mathbb R^N$ is a bounded Lipschitz domain, $\nu$ denotes the exterior normal vector field on $\partial \Omega$, 
\begin{align*}
f_q(t):= \begin{cases}
    |t|^{q-2} t,  & \text{ if } q >1,\\
    \sign(t), & \text{ if } q=1,
\end{cases}    
\end{align*}
the operator $-\Delta_p u=-\operatorname{div}(|\nabla u|^{p-2}\nabla u)$ is the well-known $p$-Laplacian (see \cite{lindqvist} for a survey), and the exponents $p$ and $q$ are subcritical in the sense that
\begin{align*}
1\leq p<N\quad \text{ and }\quad 1\leq q<p^*:=\frac{Np}{N-p}. 
\end{align*}
The extremal case $p=1$ corresponds to the $1$-Laplacian operator, formally,
\[
\Delta_1 u:= \text{div} \left(\frac{Du}{|Du|} \right),
\]
which is a nonlinear and highly degenerate operator whose study often involves tools from nonsmooth analysis and geometric measure theory.

The scalar $\lambda_{p,q}$ is called a nonlinear eigenvalue and it admits different variational characterizations depending on the values of $p$ and $q$. The limiting case $\lambda_{1,1}$ plays a central role in our analysis. In particular, our variational framework allows us to relate asymptotically minimizers of $\lambda_{p,q}$ with minimizers of a relative isoperimetric problem known as the \emph{Neumann-Cheeger problem}. This framework reveals a natural relationship between the two problems and exploits the rich geometric structure of the Neumann-Cheeger problem to describe the asymptotic shape of minimizers in domains with different characteristics, as well as to obtain results on regularity, uniqueness, multiplicity, symmetry, and symmetry breaking phenomena.

As far as we know, problem \eqref{pLap} has not been previously studied in this generality, but some particular cases and some closely related problems are available in the literature. For example, the eigenvalue problem is studied in \cite{DL26,RS16, SdLSdL, SdLSdL2, SdLSdLEJDE,gg,g01}, the Poisson and Dirichlet (linear) problems in \cite{MRS13}, and other Neumann problems that allow positive solutions can be found in \cite{BCN20}. 

We emphasize that all nontrivial solutions of \eqref{pLap} are \emph{sign changing}. Indeed, by integration by parts, if $u$ is a solution of \eqref{pLap}, then
\begin{align}\label{comp:cond}
    \int_{\Omega} \lambda_{p,q} |u|^{q-2}u \, d\mathcal{H}^N= \int_{\Omega}-\Delta_p u \, d\mathcal{H}^N
    =\int_{\partial\Omega} \partial_\nu u |\nabla u|^{p-2} \, d\mathcal{H}^{N-1}= 0,
\end{align}
and therefore any nontrivial solution $u$ must change sign. Here and in the following, $\mathcal{H}^N$ denotes the $N$-dimensional Hausdorff measure. 

Let us begin by discussing the variational characterization of the nonlinear eigenvalues $\lambda_{p,q}$ with $1<p<N$. In this case, we let 
\begin{align}\label{lambda pq intro}
\lambda_{p,q}:= \begin{cases}
\inf \left\{ \int_\Omega |D u|^p\,d\mathcal{H}^N: \, u \in W^{1,p}(\Omega), \, \int_\Omega |u|^{q-2}u \, d\mathcal{H}^N=0, \, \norm{u}_q=1 \right\},  &\hspace{-0.5cm} \text{ if } 
1< q<p^*,\vspace{0.2cm}\\
    \inf \left\{ \int_\Omega |D u|^p\, d\mathcal{H}^N: \, u \in W^{1,p}(\Omega), \, ||\{u >0 \}|-|\{ u<0 \}|| \le |\{u=0\}|, \, \norm{u}_1=1 \right\}, & \text{ if } q=1. 
\end{cases}
\end{align}
Here $W^{1,p}(\Omega)$ denotes the usual Sobolev space and $\norm{\cdot}_q$ denotes the norm in $L^q(\Omega)$. Note that, for $q=1$, one does not minimize under the condition $\int_\Omega \sign(u)d \mathcal H^N=0$, since this is not an appropriate restriction from a variational point of view. Instead, inspired by \cite{PW}, we consider the weaker condition
\begin{align}\label{M cond intro} 
\Big||\{u >0 \}|-|\{ u<0 \}|\Big| \le |\{u=0\}|.
\end{align}
Here, for $A\subset\R^N$, we use $|A|:=\mathcal{H}^N(A)$ to denote its Lebesgue measure and $\{u>0\}:=\{x\in \Omega\::\: u(x)>0\}$ denotes the corresponding super level set. The sets $\{ u<0 \}$ and $\{u=0\}$ are defined similarly.

Our first result shows that, for $1<p<N$, the nonlinear eigenvalue $\lambda_{p,q}$ is positive, it is achieved, and that the minimizers are (weak) solutions of \eqref{pLap}.

\begin{theo}\label{theo p>1 intro}
Let $1<p<N$, $q \in [1, \frac{pN}{N-p})$, and let $\lambda_{p,q}$ be given by \eqref{lambda pq intro}. 
    Then $\lambda_{p,q}>0$ is attained and any minimizer of $\lambda_{p,q}$ is a solution of \eqref{pLap} in the sense that 
    \[ \int_\Omega |Du|^{p-2} Du Dv \, d\mathcal{H}^N = \lambda_{p,q} \int_\Omega f_q(u) v \, d\mathcal{H}^N \quad \text{ for all } v \in W^{1,p}(\Omega). \]
\end{theo}

Next we turn our attention to the case $p=1$, namely, to the 1-Laplacian operator.  In this setting, the natural functional space for the variational perspective is the space of functions of bounded variation $BV(\Omega)$, since $W^{1,1}(\Omega)$ lacks the necessary compactness properties. To be more precise, let 
\begin{align*}
\lambda_{1,q}:= 
\begin{cases}
\inf \left\{|D u|(\Omega): \, u \in BV(\Omega), \, \int_\Omega |u|^{q-2}u\, d\mathcal{H}^N=0, \, \norm{u}_q=1 \right\},  & \text{ if } q \in (1, \frac{N}{N-1}),\vspace{0.2cm}\\
\inf \left\{ |Du|(\Omega): u \in BV(\Omega), \, \Big| |\{ u>0\}|-|\{u<0\}| \Big| \le |\{u=0\}|, \; \norm{u}_1=1 \right\}, & \text{ if } q=1,
\end{cases}
\end{align*}
where $|D u|$ is the total variation measure of $u$, namely, for any open set $U\subset \Omega$,
\begin{align}\label{tv:def}
|D u|(U):=\sup \left \{ \int_\Omega u \, \operatorname{div} \varphi \, d\mathcal{H}^N: \, \varphi \in C_c^\infty(U, \R^N), \, \norm{\varphi}_\infty \le 1 \right\}.    
\end{align}

In general, the minimizers of $\lambda_{1,1}(\Omega)$ do not belong to $W^{1,1}(\Omega)$, see Remark~\ref{rmk:notattained} and Corollary~\ref{head cor}, and, in Theorem~\ref{thm:head intro} below, we show the existence of a domain $\Omega$ for which all minimizers $u$ of $\lambda_{1,1}$ satisfy that $\int_\Omega \sign(u)d \mathcal H^N \ne 0$. This shows that the space $BV(\Omega)$ and the condition \eqref{M cond intro} are not just for technical convenience, but they are, in fact, essential.  

We mention that, in \cite{Chang}, a sequence of eigenvalues of $\Delta_1$ with Neumann boundary conditions is found via Ljusternik–Schnirelmann theory for nonsmooth functionals, where critical points are to be interpreted as functions with zero strong slope. Our approach  avoids the use of weak/strong slopes, and it simplifies the asymptotic study as the parameters $p$ and $q$ vary. 

Our next result is the analogue of Theorem~\ref{theo p>1 intro} for the case $p=1$. We postpone to Section~\ref{sec:p=1} the definition of weak solutions in the setting of the 1-Laplacian, since this is a bit technical at first glance.
\begin{theo}\label{theo p=1 intro}
Let $q \in [1, \frac{N}{N-1})$, then $\lambda_{1, q}>0$ is attained in $BV(\Omega)$ and
    \begin{itemize}
    \item[(i)] if $q>1$, any minimizer of $\lambda_{1, 1}$ is a solution of 
    \[ -div \left( \frac{Du}{|Du|} \right) = \lambda_{1, q} |u|^{q-2}u \quad \text{ in }\Omega, \qquad \partial_\nu \left( \frac{Du}{|Du|} \right)= 0 \text{ on }\partial \Omega,
    \]
    in the sense of Definition~\ref{def 1}. 
    \item[(ii)] if $q=1$ and $u$ is a minimizer of $\lambda_{1, 1}$, then $u$ is a solution of
    \begin{align}\label{eq lambda 11 intro}
    -div \left( \frac{Du}{|Du|} \right) = \lambda_{1, 1} s \quad \text{ in }\Omega, \qquad \partial_\nu \left( \frac{Du}{|Du|} \right)= 0 \text{ on }\partial \Omega,
    \end{align}
    in the sense of Definition~\ref{def 1}, where 
        \begin{equation}\label{selection} s(x):= \begin{cases}
        1, & \text{ if } u(x) >0,\\
 \alpha, & \text{ if } u(x) =0, \\
        -1, & \text{ if } u(x) <0, 
    \end{cases}\end{equation}
    for some $\alpha \in [-1, 1]$.
    \end{itemize}
\end{theo}

The proof of Theorem~\ref{theo p=1 intro} is again variational and it is in part inspired by \cite{KS}; however, additional technical difficulties and structural differences arise when dealing with \eqref{M cond intro} in the space $BV(\Omega)$. For instance, as detailed in Section~\ref{sec:prelim}, if $u \in BV(\Omega)$ satisfies \eqref{M cond intro}, then one can add suitable constants to $u$ and still keep the validity of \eqref{M cond intro}, see Remark~\ref{rmk: c w12}. This simple fact produces multiplicity of solutions and several technical complications. 
 
 For the 1-Laplacian, \eqref{eq lambda 11 intro} can be seen as the corresponding  homogeneous eigenvalue problem, and it also appears in \cite{SdLSdL, SdLSdL2} using Ljusternik–Schnirelman theory. A function $s$ as in \eqref{selection} is usually called a \textit{selection} of $\sign(u)$.

Our next result shows the convergence of minimizers of $\lambda_{p, q}$ as $p, q$ vary in the admissible regimes. This  complements \cite{SdLSdL, SdLSdL2}, in which minimizers of $\lambda_{p,q}$ are shown to converge to solutions to \eqref{eq lambda 11 intro} as $p=q \to 1$. Here, we 
provide a variational characterization of the limit of $\lambda_{p,q}$ and show that limiting solutions are minimizers of $\lambda_{1,1}$. 
\begin{theo}\label{theo conv intro}
Let $p_n, p \in [1, N)$, $q_n \in [1, p_nN/(N-p_n))$, and  $q \in [1, pN/(N-p))$.  Then, up to a subsequence,  
\begin{align}\label{pnqnclaim}
\lambda_{p_n,q_n} \to \lambda_{p,q} \quad \text{ if } \quad (p_n, q_n) \to (p, q). 
\end{align}
Moreover, given a minimizer $u_n$ of $\lambda_{p_n, q_n}$, there exists a minimizer $u$ of $\lambda_{p, q}$ such that, up to a subsequence,  $u_n \to u$ in $W^{1, p}(\Omega)$ if $p>1$, whereas, if $p=1$,
\begin{align}\label{pnqnclaim2}
u_n \to u \text{ in } L^r(\Omega)\quad \text{ for every } r \in [1, N/(N-1))\qquad  \text{ and }\qquad  |D u_n|(\Omega) \to |Du|(\Omega). 
\end{align}
\end{theo}
We refer to \cite{ST} for the case $p=2$, and to \cite{SST, AST, PRT, PRT2} for the study of some related nonlinear eigenvalue problems with the $p$-bilaplacian and $1$-bilaplacian operators.

In the rest of the paper, we focus our attention on the eigenvalue $\lambda_{1,1}$ and its relationship to the Cheeger-Neumann problem.

We say that a set $\omega\subset \Omega$ has finite relative perimeter in $\Omega$ if $\chi_\omega\in BV(\Omega)$, where $\chi_\omega$ denotes the characteristic function of the set $\omega$. In this case, the relative perimeter of $\omega$ in $\Omega$ is given by 
\begin{align*}
P(\omega;\Omega):=|D \chi_\omega|(\Omega).
\end{align*}
Essentially, $P(\omega;\Omega)$ is the length of the portion of the perimeter of $\omega$ that is \emph{contained in the interior of $\Omega$}.

Let 
\begin{align}\label{hdef}
h_{\mathcal N}(\Omega)&:=\inf \left\{ \frac{P(\omega; \Omega)}{|\omega|}: \, \omega\subset\Omega,\ 0 < |\omega| \le \frac{|\Omega|}{2} \right\}. 
\end{align}
The number $h_{\mathcal N}(\Omega)$ is sometimes called the \emph{Neumann-Cheeger constant} \cite{J15,B80} or the \emph{the shortest fence} \cite{BF25}. It has been used to obtain lower bounds for eigenvalues of the Neumann Laplacian. The next result establishes that $\lambda_{1,1}=h_{\mathcal N}(\Omega)$ and characterizes its minimizers. We use $u^+:=\max\{0,u\}$ and $u^-:=\max\{-u,0\}$ to denote the positive and negative parts of $u$ respectively.
\begin{theo}\label{theo charact lambda intro}
Let  $\Omega\subset \mathbb R^N$ be a bounded Lipschitz domain, then $\lambda_{1,1} = h_{\mathcal N}(\Omega).$ Moreover, 
    \begin{itemize}
        \item [$(i)$]if $u$ is a minimizer of $\lambda_{1,1}$, then $\omega_t^\pm:=\{x\in \Omega\::\: u^\pm(x)>t\}$ is a minimizer for $h_{\mathcal N}(\Omega)$ for almost every $t\in (0,\|u^\pm\|_\infty)$, whenever $\|u^\pm\|_\infty\neq 0$.
        \item [$(ii)$] if $A$ achieves $h_{\mathcal N}(\Omega)$, then $u:=\chi_{A}$ is a minimizer for $\lambda_{1,1}$. Furthermore, if $B$ also achieves $h_{\mathcal N}(\Omega)$ and $A\cap B=\emptyset$, then $u:=\alpha \chi_{A}-\beta \chi_{B}$ is a minimizer for $\lambda_{1,1}$ for any $\alpha, \beta \in \R$ such that $\alpha, \beta \ne 0$ and $\alpha\beta \ge 0$. 
    \end{itemize}
\end{theo}

To the best of our knowledge, this is the first work to establish a rigorous and robust connection between Neumann $p$-Laplacian problems and the Neumann-Cheeger constant $h_{\mathcal N}(\Omega)$ in the limit as $p \to 1$, even in the homogeneous case. This connection relies essentially on our variational characterization of $\lambda_{1,1}$ based on condition \eqref{M cond intro}, which, as far as we know, has not been previously explored in this context. We refer to Remark~\ref{ggremark} for a discussion of earlier contributions and some of their limitations.

The proof of Theorem \ref{theo charact lambda intro} is based on a standard use of the coarea formula, see, for instance, \cite{KF,L97,BF25}. We include a proof for completeness and because our definition of $\lambda_{1,1}$ differs from those considered in previous works. To be more precise, we use first an auxiliary isoperimetric problem among subsets  $\omega_1,\omega_2\subset \Omega$ satisfying
\begin{align}\label{Nomega}
\omega_1\cup \omega_2\neq \emptyset,\qquad \omega_1 \cap \omega_2=\emptyset,\qquad 
\max\{|\omega_1|,|\omega_2|\}\leq \frac{|\Omega|}{2}.    
\end{align}
In this setting, consider 
\begin{align}\label{auxprob}
\widetilde h_{\mathcal N}(\Omega):=\inf\left\{\frac{P(\omega_1;\Omega) + P(\omega_2; \Omega)}{|\omega_1 \cup \omega_2|}\::\: \omega_1,\omega_2\subset \Omega\text{ satisfy \eqref{Nomega}}\right\}.
\end{align}
Using the coarea formula, the level sets $\omega_t^\pm:=\{u^\pm>t\}$ of a minimizer $u$ of $\lambda_{1,1},$ and the variational characterization of $\lambda_{1,1},$  we show that $(\omega_t^+,\omega_t^-)$ is a minimizer of \eqref{auxprob} for almost every $t \in (0, \norm{u}_\infty)$.  Observe that \eqref{Nomega} is equivalent to 
\begin{align}\label{rem}
\omega_1\cup \omega_2\neq \emptyset,\qquad \omega_1 \cap \omega_2=\emptyset,\qquad   \Big||\omega_1|-|\omega_2|\Big|\le |\Omega\backslash (\omega_1\cup \omega_2)|,
\end{align}
and that \eqref{rem} is reminiscent of \eqref{M cond intro}. Finally, we show that $\widetilde h_{\mathcal N}(\Omega)=h_{\mathcal N}(\Omega)$ and that, if $(D^+,D^-)$ is a minimizer of $\widetilde h_{\mathcal N}(\Omega)$ such that $D^+, D^- \ne \emptyset$, then  $D^+$ and $D^-$ are minimizers for $h_{\mathcal N}(\Omega)$, see Lemma~\ref{lemma charac2}.

Our results are inspired by previous work on the Dirichlet eigenvalue problem for the 1-Laplacian, in particular \cite{KS, KF}. In that setting, the eigenvalues are directly related to the \emph{Cheeger constant}, and the corresponding minimizers are characteristic functions of \emph{Cheeger sets}, i.e., subsets that minimize the ratio between (total) perimeter and volume.  We refer to \cite{p2011} for an introduction in this topic. From this perspective, the key difference between the Dirichlet and Neumann settings is that, in the latter, only the relative perimeter (the portion of the perimeter lying in the interior of the domain) is taken into account under a volume constraint which does not fix it a priori.

To close this paper, we now characterize the minimizers of \eqref{hdef} for domains $\Omega$ with different geometric properties such as convexity or regularity. First, we show that the volume constraint in \eqref{hdef} gets saturated in smooth strictly convex domains and in cubes in any dimension. 
\begin{theo}\label{thm convex intro}
    Let $\Omega \subset \R^N$ be an open, bounded, strictly convex set with $C^{2, \alpha}$ boundary, or the cube $[0,1]^N$. 
    Then, 
\begin{align}\label{h convex intro} h_{\mathcal N}(\Omega) = \frac{2}{|\Omega|} \inf\left \{ P(E; \Omega): \, E \subset \Omega, \, |E|= \frac{|\Omega|}{2}\right \}. \end{align}
Moreover, $E$ is a minimizer of the right hand side if and only if $E$  is a minimizer for $h_{\mathcal N}(\Omega)$. 
\end{theo}

The importance of Theorem~\ref{thm convex intro} is that it transforms \eqref{hdef} into a relative isoperimetric problem \emph{with fixed volume constraint}, which has been widely studied before in the literature.  The case of two-dimensional bounded convex set is studied in \cite[Theorem 3]{C89} and the case of balls in \cite{C892}. We also mention \cite{EFKNT}, where it is shown that among convex planar domains of prescribed area, the disk uniquely maximizes the minimal relative perimeter required to split the domain into two sets of equal measure; as a consequence, among convex planar domains of prescribed area, the disk uniquely \emph{maximizes} the Neumann-Cheeger constant.

From the point of view of \eqref{pLap}, Theorem \ref{thm convex intro} can be interpreted as a symmetry breaking result (since minimizers are not radially symmetric when $\Omega$ is a ball). We also prove that this is the case for $N$-dimensional annuli and we give a complete characterization of radially symmetric minimizers, see Lemma~\ref{rad annulus}.  

Another important consequence of Theorem~\ref{thm convex intro} is that, in that setting, any minimizer of $\lambda_{1,1}$ is of the form $u= \alpha \chi_{\omega_1} - \beta \chi_{\omega_2}$ for $\alpha, \beta \ge 0$, $(\alpha, \beta) \ne (0,0)$, and $\omega_1, \omega_2$ minimize \eqref{h convex intro}, see Corollary~\ref{coro convex}. In particular, any sign-changing minimizer $u$ of $\lambda_{1, 1}$ satisfies
    $\int_\Omega \sign(u) d \mathcal H^N=0$. 
    However, there also exist one-sign minimizers of $\lambda_{1,1}$, showing that there are minimizers with $\int_\Omega \sign(u)\ne0$ and satisfying \eqref{eq lambda 11 intro} with $s \ne \sign(u)$, see also Remark~\ref{rmk:not sign}.

A natural question is whether the volume constraint in \eqref{hdef} is \emph{always} saturated; namely, if \eqref{h convex intro} \emph{always} holds. Our next result shows, via an explicit example, that this is not true. 
\begin{theo}\label{thm:head intro}
    There exists a domain $\Omega\subset \R^2$ such that
    \begin{itemize}
        \item [$(i)$] $h_{\mathcal N}(\Omega)$ has a unique minimizer $\omega$ and $|\omega|<|\Omega|/2$.
        \item [$(ii)$] \emph{all} the minimizers $u$ of $\lambda_{1,1}(\Omega)$ do not change sign and satisfy that $|\{ u=0\}| \ne 0.$ In particular, $\int_\Omega  \sign(u) \, d\mathcal{H}^N\ne 0$ and the unique continuation principle does not hold for \eqref{eq lambda 11 intro}.
    \end{itemize}    
\end{theo}

Observe that Theorem \ref{thm:head intro} $(ii)$, together with Theorem \ref{theo conv intro}, implies that minimizers of $\lambda_{p,q}$ with $q>1$\textemdash which are always sign-changing\textemdash necessarily become nonnegative (or nonpositive) in the limit as $(p,q)\to (1,1)$, whenever $\Omega$ is the two-dimensional domain constructed in Theorem~\ref{thm:head intro} (see Figure~\ref{hwe}).  This domain $\Omega$ consists of three squares of different sizes connected by a thin tube. The underlying mechanism is that a minimizer $\omega$ of $h_{\mathcal N}(\Omega)$ (shown as the shaded region in Figure~\ref{hwe}) exploits the thin tube by concentrating the interior portion of its perimeter there while still enclosing a comparatively large area.

\begin{figure}[h!]
    \centering
    \includegraphics[width=0.95\linewidth]{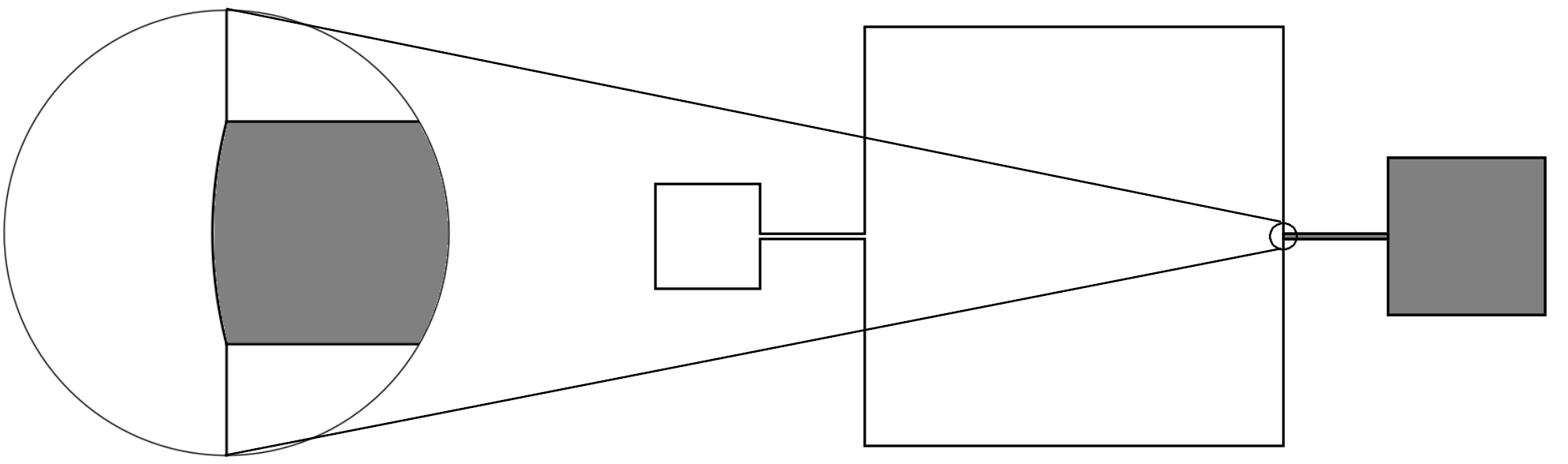}
    \caption{A domain (right) for which $h_{\mathcal N}(\Omega)$ has a unique minimizer given by the shaded region.}
    \label{hwe}
\end{figure}

The proof of Theorem~\ref{thm:head intro} relies heavily on regularity and rigidity results for minimizers of $h_{\mathcal N}(\Omega)$, that we show using some geometric measure theory tools.  To be more precise, we show the following.
\begin{theo}\label{regn2}
Let $\Omega\subset \mathbb R^2$ be a Lipschitz bounded domain. Let $D$ be a minimizer for 
\begin{align*}
h_{\mathcal N}(\Omega)&:= \inf \left\{ \frac{P(\omega; \Omega)}{|\omega|}:\,\omega\subset \Omega, \  0 <|\omega| \le \frac{|\Omega|}{2} \right\}
\end{align*}
with $|D|<|\Omega|/2$. Then any portion of the boundary $\partial D$ is part of a circle with radius $r:=\frac{1}{\lambda_{1,1}}>0$. Moreover, if $x_0\in \partial D\cap \partial \Omega$ and $\partial  \Omega$ is locally $C^2$ around $x_0$, then $\partial D$ meets $\partial \Omega$ orthogonally at $x_0$.
\end{theo}

Theorem \ref{regn2} is known to experts and it is essentially contained, for instance, in \cite[Theorem 2]{C89}. We include a different proof in Section \ref{reg:app} for completeness. 

Using these structural properties, we perform a case-by-case analysis and conclude that the configuration shown in Figure~\ref{hwe} is the only possibility for the minimizer. In Section~\ref{counterex}, we also characterize the case where the two squares on the sides (see Figure \ref{hwe}) have the same length (which leads to multiplicity of minimizers) and, in Section~\ref{reg:app}, we also prove a finer regularity result for minimizers of $h_{\mathcal N}(\Omega)$ that considers any dimension $N\geq 2$. We show that, if $\omega$ is a minimizer of $h_{\mathcal N}(\Omega)$, then $\partial\omega\cap \Omega$ does not develop singularities in dimensions $2\leq N\leq 7$ (i.e., $\partial\omega\cap \Omega$ is smooth), whereas for $N\geq 8$ singularities can appear and we establish a bound on the Hausdorff dimension of the singular set, see Theorem~\ref{thm reg}. In the proof we show that minimizers fit into the framework of volume-constrained perimeter minimizers, so that standard geometric measure theory regularity results apply; the only additional subtlety arises in the half-volume case, where the variational argument is necessarily one-sided.

\medskip

This paper is organized as follows. In Section~\ref{sec:prelim}, we establish our variational framework and establish existence results, treating both the cases $p>1$ and $p=1$. Section~\ref{sec convergence} is devoted to the asymptotic behavior of $\lambda_{p,q}$ and the corresponding minimizers as the parameters vary. In Section~\ref{geom char}, we provide a geometric characterization of $\lambda_{1,1}$ in terms of a relative isoperimetric problem and prove some regularity and rigidity results for minimizers. Finally, in Section \ref{sec:minimizers} we study the exact shape of minimizers in convex, symmetric, and asymmetric domains. 

\subsection{Notation}\label{ss:notation}

Let $\Omega$ be an open, bounded, Lipschitz domain. The space of functions of bounded variation is denoted by $BV(\Omega)$, namely, $u \in BV(\Omega)$ if $u \in L^1(\Omega)$ and the total variation $|Du|(\Omega)< + \infty$ (see \eqref{tv:def}).  

We endow $BV(\Omega)$ with the norm
\begin{align*}
    \|u\|_{BV(\Omega)}:=\|u\|_{1}+|Du|(\Omega)
\end{align*}
With this norm, $BV(\Omega)$ is a Banach space. Furthermore, if $(u_n)\subset BV(\Omega)$ is a sequence such that 
\begin{align}\label{comp:em}
\sup_{n\in \mathbb N}\|u_n\|_{BV(\Omega)}<\infty,\text{ then, up to a subsequence, $u_n\to u$ in $L^q(\Omega)$ for $q \in [1, N/(N-1))$ and $u\in BV(\Omega)$.\quad }
\end{align}
The embedding $BV(\Omega) \hookrightarrow L^q(\Omega)$ is compact for $q \in [1, N/(N-1))$ and $BV(\Omega) \hookrightarrow L^\frac{N}{N-1}(\Omega)$ is continuous, see \cite[Corollary 3.49]{AFP} and \cite[Section 5.2.3]{evansgariepy}.

Next, let us recall the notion of \emph{normal trace}, which allows us to have a generalized integration by parts formula. We use 
$\mathcal{H}^{N}, \mathcal{H}^{N-1}$ to denote the $N$- and $(N-1)$-dimensional Hausdorff measure, respectively.

\begin{lemma}[Proposition 6.6 in \cite{KS}]\label{anzellotti}
    Let $\Omega \subset \R^N$ be a bounded, Lipschitz domain.  For any $z \in L^\infty(\Omega, \R^N)$ such that $\operatorname{div}(z) \in L^1(\Omega)$, there exists a function $[z, \nu] \in L^\infty(\partial \Omega)$ such that 
    \begin{align*}
        \int_\Omega \varphi \operatorname{div}(z) \,d\mathcal{H}^N+ \int_\Omega z \cdot D \varphi \,d\mathcal{H}^N= \int_{\partial \Omega} [z, \nu] \varphi \, d\mathcal{H}^{N-1} \quad \text{ for all }\varphi \in C^\infty(\overline \Omega),
    \end{align*}
    where $\nu$ is the unit inner normal to $\Omega$. The function $[z, \nu]$ is called the normal trace of $z$.
\end{lemma}
If $z \in H^2(\Omega)$, then $\partial_\nu z=[z,\nu]$.

\section{Variational framework and existence results}\label{sec:prelim}

\subsection{Preliminaries}

Every solution of \eqref{pLap} satisfies \eqref{comp:cond}. Therefore, as a first step in our analysis we recall some results from \cite{PW} regarding \eqref{comp:cond} and how this extends in the BV setting. 
\begin{lemma}[Lemma 2.1 in \cite{PW}]\label{lemma c pw}
Let $q>1$. For every $u \in L^1(\Omega)$ there exists only one value $c_q(u) \in \R$ such that $\int_\Omega |u+c_q(u)|^{q-2} (u+c_q(u))=0$. Moreover the map $c_q: L^1(\Omega) \to \R$ is continuous, and
\[ \inf_{c \in \R} \norm{u+c}_q = \norm{u+c_q(u)}_q. \]
\end{lemma}

This Lemma is not true in general if $q=1$. In this case, we consider the set 
\begin{equation}\label{N omega}
\mathfrak {N}(\Omega):= \left\{ u \in BV(\Omega), \, \Big| |\{ u>0\}|-|\{u<0\}| \Big| \le |\{u=0\}| \right\}. 
\end{equation}

Notice that the condition in $\mathfrak {N}(\Omega)$ can also be written as follows
\begin{equation}\label{charact M}  
\int_\Omega \sgn_-(u)\,d\mathcal{H}^N \le 0 \le \int_\Omega \sgn_+ (u)\,d\mathcal{H}^N, 
\end{equation}
where
\begin{align}\label{sgnpm}
    \sgn_-(t):=\mathds{1}_{t > 0} - \mathds{1}_{t \le 0}\qquad \text{ and }\qquad \sgn_+(t):=\mathds{1}_{t \ge 0} - \mathds{1}_{t < 0}. 
\end{align} 

The set $\mathfrak N (\Omega)$ has been considered in \cite{PW} for functions in $W^{1, 2}(\Omega)$. When taking into account $BV$ functions, deep differences arise; in particular, one loses uniqueness.
\begin{lemma}\label{c BV}
\begin{itemize}
\item[(i)]    If $u \in \mathfrak {N}(\Omega)$, then
    \[ \int_\Omega |u+c|\,d\mathcal{H}^N \ge \int_\Omega |u|\,d\mathcal{H}^N \quad \text{ for every $c \in \R$. }\]
\item[(ii)] For every $u \in BV(\Omega)$ there exists $c(u) \in \R$ such that $u+ c(u) \in \mathfrak N(\Omega)$. 
\item[(iii)] Define
\[ \underline c(u):=\inf\{c: u+ c\in \mathfrak {N}(\Omega)\}, \quad \overline c(u):=\sup\{c: u+ c\in \mathfrak {N}(\Omega)\}. \]
Let $u \le v$. Then, $\underline c (u) \ge \underline c (v)$ and $\overline c (u) \ge \overline c(v)$. 
\item[(iv)] The map in $BV(\Omega)$ defined as $u \mapsto \overline c(u)$ is upper semicontinuous, whereas $u \mapsto \underline c(u)$ is lower semicontinuous. 
\end{itemize}
\end{lemma}
\begin{proof}
(i) 
Let  $u \in \mathfrak {N}(\Omega)$ and consider first $c<0$. If $u \le 0,$ then $\int_\Omega |u+c|\,d\mathcal{H}^N > \int_\Omega |u|\,d\mathcal{H}^N$. Suppose now $u^+ \ne 0$. If $|\{ 0 < u < |c|\}|>0$ then we can proceed as in \cite[Lemma 5.1]{PW} to prove that $\int_\Omega |u+c| \,d\mathcal{H}^N> \int_\Omega |u|\,d\mathcal{H}^N$. Thus we are left with the case
\begin{align*}
|\{ 0 < u < |c|\}| =0.
\end{align*}
Let
\begin{align*}
\Omega_0:=\{ 0 < u < |c|\},\qquad
\Omega_1:=\{ u < 0\},\quad 
\Omega_2:=\{ u = 0\},\quad 
\Omega_3:=\{ u = |c| \},\quad 
\Omega_4:=\{ u > |c|\}.
\end{align*}
Then $\Omega = \bigcup\limits_{i=0}^4 \Omega_i$.  Since $u \in \mathfrak{N}(\Omega)$ and $|\Omega_0|=0$, by \eqref{N omega} we deduce that 
\begin{equation}\label{ineq sets} 
\abs{ |\Omega_3|+|\Omega_4|-|\Omega_1|} \le |\Omega_2|. 
\end{equation}
Moreover, since we are assuming that $c<0$,
\[ \int_\Omega |u|\,d\mathcal{H}^N= - \int_{\Omega_1} u\,d\mathcal{H}^N + \int_{\Omega_4} u\,d\mathcal{H}^N -c |\Omega_3|. \]
On the other hand,
\[ u+c= \begin{cases}
    u+c & \text{ in } \Omega_1\cup \Omega_4,\\
    c & \text{ in } \Omega_2,\\
    0 & \text{ in } \Omega_3.
\end{cases}\]
Hence, by \eqref{ineq sets}, we have that $|\Omega_2| + |\Omega_1| \ge |\Omega_3|+|\Omega_4|$, and therefore
\begin{align*}
\int_\Omega |u+c|\,d\mathcal{H}^N
& = -c|\Omega_2| + \int_{\Omega_4} u\,d\mathcal{H}^N - \int_{\Omega_1} u\,d\mathcal{H}^N + c |\Omega_4| - c |\Omega_1|\\
&= (-c)(|\Omega_2|-|\Omega_4|+|\Omega_1|-|\Omega_3|)+ \int_{\Omega} |u|\,d\mathcal{H}^N
\geq \int_\Omega |u|\,d\mathcal{H}^N.
\end{align*}

The case $c >0$ can be treated similarly. 

\medskip

(ii) The proof follows as in \cite[Lemma 5.1(iv)]{PW}. 

\medskip

(iii) We first show that $u + \underline c(u) \in \mathfrak {N}(\Omega)$. Indeed, let $c_n$ be such that $u+c_n \in \mathfrak {N}(\Omega)$ and $c_n \to \underline c(u)$. By \eqref{charact M}, $\int_\Omega \sgn_-(u+ c_n)\,d\mathcal{H}^N \le 0 \le \int_\Omega \sgn_+ (u+c_n)\,d\mathcal{H}^N. $ By dominated convergence, 
\begin{align}\label{s1}
\int_\Omega \sgn_-(u+ \underline c(u))\,d\mathcal{H}^N \le 0 \le \int_\Omega \sgn_+ (u+\underline c(u))\,d\mathcal{H}^N;    
\end{align}
namely, $u+ \underline c(u) \in \mathfrak {N}(\Omega)$.

Let $u \le v$ and assume by contradiction that $\underline c(u) < \underline c(v)$. Using \eqref{sgnpm} and \eqref{s1}, we have
\begin{align*}
\int_\Omega \sgn_-(v+\underline c(u))\,d\mathcal{H}^N \le \int_\Omega \sgn_-(v +\underline c(v))\,d\mathcal{H}^N \le 0
\end{align*}
and
\begin{align*}
\int_\Omega \sgn_+(v +\underline c(u ))\,d\mathcal{H}^N \ge  \int_\Omega \sgn_+(u+\underline c(u))\,d\mathcal{H}^N \ge 0. 
\end{align*}
Thus, $v+ \underline c(u) \in \mathfrak {N}(\Omega)$ and $\underline c (u) \ge \underline c(v)$ by definition of $\underline c(v)$, a contradiction. 

The claim for $\overline c(u)$ can be argued similarly.

\medskip

(iv) Let $u_n \to u$ in $BV(\Omega)$. Up to a subsequence, we can assume that $u_n \to u$ pointwise. 
Let us consider the map $\overline c(u)$ (the other case being similar). We first show that $\overline c(u_n)$ is bounded. Assume for the sake of contradiction that, passing to a subsequence, $\overline c(u_n) \to +\infty$. Then, by \eqref{charact M}, 
\[ 0 \ge \int_\Omega \sgn_-(u_n + \overline c (u_n))\,d\mathcal{H}^N \to |\Omega|, \]
which yields a contradiction. Similarly, if $\overline c(u_n) \to -\infty$, then $0 \le \int_\Omega \sgn_+(u_n + \overline c (u_n))\,d\mathcal{H}^N \to -|\Omega|,$ which again yields a contradiction. 

Since $\overline c(u_n)$ is bounded, it converges, up to a subsequence, to some $c \in \R$. By \eqref{charact M},
\[ \int_\Omega \sgn_-(u_n+ \overline c(u_n))\,d\mathcal{H}^N \le 0 \le \int_\Omega \sgn_+ (u_n+ \overline c(u_n))\,d\mathcal{H}^N, \]
and, using that $u_n \to u$ a.e. in $\Omega$, $\overline c(u_n) \to c$, and dominated convergence, we obtain that
\[ 
\int_\Omega \sgn_-(u+c)\,d\mathcal{H}^N \le 0 \le \int_\Omega \sgn_+ (u+c)\,d\mathcal{H}^N.
\]
This implies that $u+c \in \mathfrak {N}(\Omega)$. Therefore, $c \le \overline c(u)$, as claimed.
\end{proof}

\begin{remark}\label{rmk: c w12}
    If $u \in W^{1, 2}(\Omega)$, then $\overline{c}(u)=\underline c (u)$ and it is continuous, see \cite[Lemma 5.1]{PW}. 
\end{remark}
\begin{remark}\label{remark c BV}
  We cannot exclude that for some $c \in \R$ one has
  \begin{equation}\label{equality c} \int_\Omega |u+c|\,d\mathcal{H}^N= \int_\Omega |u|\,d\mathcal{H}^N, \end{equation}
and, in particular, we do not have uniqueness of $c(u)$ in item $(ii)$ of Lemma~\ref{c BV}. This contrasts with the case of functions $W^{1,2}(\Omega)$, see \cite[Lemma 5.1]{PW}.
For instance, let $\Omega :=(0, 1)$ and $u:=\mathds{1}_{t > \frac{1}{2}}$. Then $u \in BV(\Omega)$ and 
\[ \int_0^1 |u+c|\,d\mathcal{H}^N= \int_0^1 |u|\,d\mathcal{H}^N \quad \text{ if and only if } c \in [-1,0 ]. \]
Hence, $\underline c(u) =-1$ and $\overline c (u) =0$.

Notice that, if $u \in \mathfrak {N}(\Omega)$ is such that \eqref{equality c} happens for some $c<0$, then, from the proof of Lemma~\ref{c BV},  necessarily 
\[ |\{ 0 < u < |c|\}|=0\]
and equality must hold in \eqref{ineq sets}. Then, if $|\{ u >0 \}| \ge |\{ u< 0\}|$, we have that
\[ |\{u \le 0\}| = |\{u=0\}| + |\{ u <0\}| = |\{ u= |c|\} |+| \{ u > |c|\}|= |\{ u > 0\}|.\]
On the other hand, if $|\{ u >0 \}| \le  |\{ u< 0\}|$,
\[ |\{u <0 \}|= |\{u \ge 0\}|. \]
In particular, the condition in $\mathfrak {N}(\Omega)$ is satisfied with equality. 

Similarly, if $u \in \mathfrak {N}(\Omega)$ is such that \eqref{equality c} happens for some $c>0$, then $|\{ -c < u < 0\}|=0$ and $||\{u < 0\}|-|\{u>0\}||=|\{u=0\}|.$
\end{remark}
\begin{remark}\label{rmk equality c}
    If $c_1, c_2$ are such that $u+c_1$ and $u+c_2$ are both in $\mathfrak {N}(\Omega)$, then $\norm{u+c_1}_1=\norm{u+c_2}_1.$
\end{remark}

\subsection{The case \texorpdfstring{$p>1$}{p>1}}\label{sec:p>1}
For $p \in (1, N)$ and $q \in [1, \frac{pN}{N-p})$, let
\begin{equation}\label{lambda pq def} 
\lambda_{p,q}:= 
\begin{cases} \displaystyle \inf \left\{ \int_\Omega |D u|^p\,d\mathcal{H}^N: \, u \in W^{1,p}(\Omega), \, \int_\Omega |u|^{q-2}u\,d\mathcal{H}^N=0, \, \norm{u}_q=1 \right\}, & \text{ if } q>1,\\
\displaystyle \inf \left\{ \int_\Omega |D u|^p\,d\mathcal{H}^N: \, u \in W^{1,p}(\Omega), \, \abs{|\{ u>0\}|-|\{u < 0\}|}\le |\{u=0\}|, \, \norm{u}_1=1 \right\}, & \text{ if } q=1. 
\end{cases} \end{equation}

\begin{lemma}\label{lambda att pq}
The quantity $\lambda_{p,q}>0$ is attained. 
\end{lemma}
\begin{proof}
Let  $p \in (1, N)$, $q \in [1, \frac{pN}{N-p})$, and let $u_n\in W^{1, p}(\Omega)$ be a minimizing sequence for $\lambda_{p,q}$. By Sobolev embeddings, up to a subsequence, there exists $u \in W^{1, p}(\Omega)$ such that $u_n \to u$ in $L^q(\Omega)$. If $q>1$, we immediately deduce $\int_\Omega |u|^{q-2}u d \mathcal H^N=0$, and $u$ is a minimizer for $\lambda_{p, q}$. So, assume that $q=1$. We adapt some arguments in \cite[Lemma 5.3]{PW}.
Since $u_n \in \mathfrak {N}(\Omega)$, by Lemma~\ref{c BV}, we have
\[ \int_\Omega |u+c|\,d\mathcal{H}^N = \lim_{n \to \infty} \int_\Omega |u_n + c|\,d\mathcal{H}^N \ge \lim_{n \to \infty} \int_\Omega |u_n|\,d\mathcal{H}^N = \int_\Omega |u| \,d\mathcal{H}^N \quad \text{ for all } c \in \R. \]
Let $\tilde c$ be such that $u+ \tilde c \in \mathfrak {N}(\Omega)$. By Lemma~\ref{c BV}, $\int_\Omega |u +\tilde c +c|\,d\mathcal{H}^N \ge \int_\Omega |u +\tilde c |\,d\mathcal{H}^N \ge \int_\Omega |u|\,d\mathcal{H}^N$ for any $c \in \R$. Choose $c= - \tilde c$. Then, $\int_\Omega |u |\,d\mathcal{H}^N \ge  \int_\Omega |u +\tilde c |\,d\mathcal{H}^N \ge \int_\Omega |u |\,d\mathcal{H}^N,$ which implies that $\int_\Omega |u +\tilde c |\,d\mathcal{H}^N =\int_\Omega |u|\,d\mathcal{H}^N.$ Since $u  \in W^{1, p}(\Omega)$, Remark~\ref{rmk: c w12} implies that $u \in \mathfrak{N} (\Omega)$, and the claim follows. 
\end{proof}

\begin{prop}\label{prop:sol}
Let $p \in (1, N)$ and $q \in (1, \frac{pN}{N-p})$. Any minimizer of $\lambda_{p, q}$ is a solution of \eqref{pLap} in the sense that
 \begin{align}\label{wf}
     \int_\Omega |Du|^{p-2} Du Dv \,d\mathcal{H}^N= \lambda_{p, q} \int_\Omega |u|^{q-2} u v\,d\mathcal{H}^N \quad \text{ for all } v \in W^{1,p}(\Omega). 
     \end{align}
\end{prop}
\begin{proof}
We follow \cite[Lemma 2.2]{PW}.  Notice that the functional $\varphi:W^{1, p}(\Omega)\backslash\{0\}\to \mathbb R$ given by $\varphi(v)= \frac{\norm{Dv}_p}{\norm{v}_q}$ is of class $C^1$ and 
\begin{align}\label{derv}
\varphi'(u)v
=
\frac{
\|u\|_{q}^{q}
\int_\Omega |Du|^{p-2} Du Dv \,d\mathcal{H}^N
-
\|Du\|_{p}^{p}
\int_\Omega |u|^{q-2} u v \,d\mathcal{H}^N
}{
\|u\|_{q}^{q+1}\,\|Du\|_{p}^{p-1}
},\qquad 
u,v\in W^{1, p}(\Omega)\backslash\{0\}.
\end{align}
Let $u \in W^{1, p}(\Omega)$ be the minimizer of $\lambda_{p, q}$ with $\norm{u}_q=1$; in particular, $\lambda_{p, q}=\norm{Du}^p_p$ and $\int_\Omega |u|^{q-2}u\,d\mathcal{H}^N=0$. If $\varphi'(u)=0$, then \eqref{derv} implies \eqref{wf}.

Assume, by contradiction, that there is $v \in W^{1, p}(\Omega)$ such that $\varphi'(u)v <0$. Then there exists $\varepsilon>0$ such that, for every $w \in W^{1,p}(\Omega)$ with $\norm{w}_{W^{1,p}} < \varepsilon$ and every $t \in (0, \varepsilon)$,
\begin{align}\label{ineq1}
    \varphi(u+w+tv)\le \varphi(u+w)-\varepsilon t.
\end{align}

By the continuity of $c_q(\cdot)$ (see Lemma~\ref{lemma c pw}) and using that $c_q(u)=0$, there exists $t>0$ such that the constant function $w:= c_q(u+tv)$ satisfies $\norm{w}_{W^{1,p}} < \varepsilon$. Then, by \eqref{ineq1},
\[ \frac{\left(\int_\Omega |D (u+tv+ c_q(u+tv))|^p \,d\mathcal{H}^N\right)^{1/p}}{(\int_\Omega |u+tv+ c_q(u+tv)|^q\,d\mathcal{H}^N)^{1/q}} \le \frac{\norm{Du}_p }{\norm{u+ c_q(u+tv)}_q} - \varepsilon t.  \]
By  Lemma~\ref{lemma c pw}, $\norm{u+c}_q \ge \norm{u}_q=1$ for all $c \in \R.$ Therefore,
\[ \frac{\left(\int_\Omega |D (u+tv+ c_q(u+tv))|^p \,d\mathcal{H}^N\right)^{1/p}}{(\int_\Omega |u+tv+ c_q(u+tv)|^q\,d\mathcal{H}^N)^{1/q}} \le \lambda_{p, q}^{\frac 1p} - \varepsilon t < \lambda_{p, q}^{\frac 1p},  \]
contradicting the optimality of $\lambda_{p, q}$. Thus, 
$ \varphi'(u)=0$ and this ends the proof. 
\end{proof}

\begin{remark}
Note that Proposition~\ref{prop:sol} does not follow from the Lagrange multiplier rule, because the restriction in \eqref{lambda pq def}  may not be a $C^1$-manifold.
\end{remark}

\begin{prop}\label{prop:sol1}
Let $p\in (1, N)$. 
    Any minimizer of $\lambda_{p, 1}$ is a solution of 
    \begin{align*}
        -\Delta_p u= \lambda_{p, 1} \sgn(u) \quad \text{ in }\Omega,\qquad 
        \partial_\nu u=0 \quad \text{ on }\partial \Omega, 
    \end{align*}
    in the sense that
    \[ \int_\Omega |Du|^{p-2} Du Dv \,d\mathcal{H}^N= \lambda_{p, 1} \int_\Omega \sgn(u) v\,d\mathcal{H}^N \quad \text{ for all } v \in W^{1,p}(\Omega). \]
\end{prop}
\begin{proof}
Let $u \in \mathfrak N(\Omega)$ be a minimizer of $\lambda_{p, 1}$. We follow \cite{PW} (see also \cite{SST}). We claim that 
\begin{equation}\label{claim min is a solution} 
\int_\Omega \sgn_-(u) v\,d\mathcal{H}^N \le \norm{D u}_p^{-p} \int_\Omega |D u |^{p-2} D u D v \,d\mathcal{H}^N \le \int_\Omega \sgn_+ (u) v \,d\mathcal{H}^N\quad \text{ for any } v \in W^{1,p}(\Omega), \, v \ge 0. 
\end{equation}
We only show the first inequality, since the second one can be proved similarly. By contradiction, assume there exists $v \in W^{1, p}(\Omega)$, $v \ge 0$, such that 
\begin{align*}
\int_\Omega \sgn_-(u) v\,d\mathcal{H}^N > \norm{D u}_p^{-p} \int_\Omega |D u |^{p-2} D u D v\,d\mathcal{H}^N. 
\end{align*}
As in \cite{PW}, for any $c \le 0$ and $t \ge 0$,
\[
\norm{u+c+tv}_1 \ge 1 +t \int_\Omega \sgn_-(u) v\,d\mathcal{H}^N.  
\]
Moreover, recalling Lemma~\ref{c BV}, one has that $\underline c(u+tv) \le \underline c (u) \le 0$ if $v \ge 0$.  Thus, 
\begin{align*}
    &\frac{\left(\int_\Omega |D (u+tv+ \underline c(u+tv))|^p \,d\mathcal{H}^N\right)^{1/p}}{\int_\Omega |u+tv+ \underline c(u+tv)|\,d\mathcal{H}^N} \le \frac{\left(\int_\Omega |D (u+tv)|^p\,d\mathcal{H}^N\right)^{1/p} }{ 1 +t \int_\Omega \sgn_-(u) v\,d\mathcal{H}^N} \\
    &\qquad =\lambda_{p, 1}^{1/p} + t\left( \left(\int_\Omega |D u|^p\,d\mathcal{H}^N\right)^{1/p-1} \int_\Omega |D u|^{p-2} D u D v\,d\mathcal{H}^N - \norm{D u}_p  \int_\Omega \sgn_-(u) v \,d\mathcal{H}^N\right) + o(t) \\
    & \qquad =\lambda_{p, 1}^{1/p}+ t \norm{D u}_p\left(\norm{D u}_p^{-p}\int_\Omega |D u|^{p-2} D u D v \,d\mathcal{H}^N- \int_\Omega \sgn_-(u) v\,d\mathcal{H}^N \right) + o(t).  
\end{align*}
The second term in the last line is negative by the contradiction assumption. Therefore,
\[ 
\frac{\left(\int_\Omega |D (u+tv+ \underline c(u+tv))|^p \,d\mathcal{H}^N\right)^{1/p}}{\int_\Omega |u+tv+ \underline c(u+tv)|\,d\mathcal{H}^N} < \lambda_{p, 1}^{1/p}
\qquad \text{for $t$ is small enough.}
\]
 This yields a contradiction, and therefore \eqref{claim min is a solution} holds. Then, 
\begin{align*}
\lambda_{1,p}^{-1} \, \int_\Omega |D u |^{p-2} D u D v \,d\mathcal{H}^N \le \int_\Omega \left[ \sgn_+(u) v^+ - \text{ sgn}_-(u) v^-\right] \,d\mathcal{H}^N \le \norm{v}_1\qquad \text{for any $v \in W^{1, p}(\Omega)$}. 
\end{align*}
As a consequence, the function $f(u): C_c^\infty(\Omega) \to \R$ given by $f(u) v= \lambda_{1,p}^{-1}\int_\Omega |D u |^{p-2} D u D v\,d\mathcal{H}^N$
is continuous in the $L^1(\Omega)$ norm. Thus, it can be uniquely extended to an element of $(L^1(\Omega))'$, and there exists $\eta \in L^\infty(\Omega)$ such that 
\[ f(u) v= \int_\Omega \eta v\,d\mathcal{H}^N \quad \text{ for all } v \in L^1(\Omega). \]
By \eqref{claim min is a solution}, $\sgn_-(u) \le  \eta \le \sgn_+(u)$ a.e. in $\Omega.$ Notice that, in the set $\{ u=0\},$ one has $D u=0$ and, by \cite{Lou}, this implies that $\eta=0$ a.e. in $\{ u=0\}$.
Therefore, $\eta = \sgn(u)$ a.e. in $\Omega$ and the claim follows.
\end{proof}

We are now ready to give the  proof of Theorem~\ref{theo p>1 intro}. 
\begin{proof}[Proof of Theorem~\ref{theo p>1 intro}]
    The claim follows from Lemma~\ref{lambda att pq}, Proposition~\ref{prop:sol}, and Proposition~\ref{prop:sol1}. 
\end{proof}

\begin{remark}

    In this paper, we focus on weak solutions, but let us briefly recall some regularity results in this setting. Weak solutions are known to be bounded via a Moser iteration argument (see \cite{MarinoWinkert}). As a consequence, by \cite{tolksdorf}, they belong to $C^{1,\alpha}(\Omega)$, and this regularity extends up to the boundary provided that $\partial\Omega$ is of class $C^{1,\alpha}$, see \cite{lieberman}. We also emphasize that unique continuation for the $p$-Laplacian remains a difficult and largely open problem, see \cite{GM}.
\end{remark}

 \subsection{The case \texorpdfstring{$p=1$}{p=1}}\label{sec:p=1}

We first give some definitions. 
\begin{defi}\label{def 1}
Let $t \ge 1$, $f \in L^t(\Omega)$, $g \in L^   \infty(\partial \Omega)$, and $z \in L^\infty(\Omega, \R^N)$. 
We say that $z$ is a solution of 
    \begin{align*}
    -\operatorname{div} z= f\quad  \text{ in } \Omega,\qquad \partial_\nu z=g\quad \text{ on } \partial \Omega, 
    \end{align*}
if  
    \[ \int_\Omega z \cdot D \varphi\,d\mathcal{H}^N - \int_{\partial \Omega} g \varphi\, d\mathcal{H}^{N-1} = \int_\Omega f \varphi \,d\mathcal{H}^N \quad \text{ for all } \varphi \in C^\infty(\overline \Omega).\]
\end{defi}
\begin{remark}
    This means that $g=[z, \nu]$ in the sense of Lemma~\ref{anzellotti}.
\end{remark}

\begin{defi}\label{def solution}
Let $f \in L^{N}(\Omega)$ be such that $\int_\Omega f\,d\mathcal{H}^N=0$.  We say that $u \in BV(\Omega)$ is a solution of
\begin{equation}\label{def sol}
- \operatorname{div}\left( \frac{D u}{|Du|} \right) = f \quad  \text{ in } \Omega, \qquad
\partial_\nu \left( \frac{D u}{|Du|} \right)=0 \quad  \text{ on } \partial \Omega,
\end{equation}
if there exists $z \in L^\infty(\Omega, \R^N)$  such that 
\begin{itemize}
    \item[(i)] $\norm{z}_\infty \le 1$.
    \item[(ii)] $z$ is a solution of $-\operatorname{div} z= f$ in $\Omega$ with $\partial_\nu z=0$ on $\partial \Omega,$ in the sense of Definition~\ref{def 1}; namely,
    \begin{align}\label{def:eq}
    \int_\Omega z \cdot D \varphi \,d\mathcal{H}^N= \int_\Omega f \varphi \,d\mathcal{H}^N \quad \text{ for any } \varphi \in C^\infty(\overline \Omega).    
    \end{align}
    \item[(iii)] 
    $ |D u|(\Omega) = - \int_\Omega \operatorname{div} z \, u \,d\mathcal{H}^N= \int_\Omega f u \,d\mathcal{H}^N$.
\end{itemize}
\end{defi}

\subsubsection{Case \texorpdfstring{$q>1$}{q>1}}

Let $q \in (1, N/(N-1))$, and set 
\begin{align} \nonumber \lambda_{1, q}:=& \inf \left \{ \frac{|D u|(\Omega)}{\norm{u}_q} : u \in BV(\Omega) \setminus \{ 0\}, \, \int_\Omega |u|^{q-2}u\,d\mathcal{H}^N=0 \right \} \\
\label{lambda 1 q}= & \inf \left \{ |D u|(\Omega) : u \in BV(\Omega), \, \int_\Omega |u|^{q-2}u\,d\mathcal{H}^N=0, \, \norm{u}_q=1 \right \}. \end{align}

\begin{lemma}\label{att}
    Let  $q \in (1, N/(N-1))$. The quantity $\lambda_{1, q}$ is attained. 
\end{lemma}
\begin{proof}
Let $u_n \in BV(\Omega)$ be a minimizing sequence for $\lambda_{1, q}$ with $\norm{u_n}_q=1$. By \eqref{comp:em}, there is $u \in BV(\Omega)$ and a subsequence $u_j$ such that $u_j \to u$ in $L^q(\Omega)$. Thus, $\int_\Omega |u|^{q-2}u\,d\mathcal{H}^N=0$, $\norm{u}_q=1$, and $\lambda_{1, q}$ is attained.
\end{proof}

\begin{lemma}\label{lambda1qlem}
For  $q \in (1, N/(N-1))$,
\begin{equation}\label{equivalent lambda}
\lambda_{1, q}= \inf \left\{ \int_\Omega |D u|\,d\mathcal{H}^N: \, u \in W^{1,1}(\Omega), \int_\Omega |u|^{q-2}u\,d\mathcal{H}^N=0, \norm{u}_q=1 \right\}. \end{equation}
\end{lemma}
\begin{proof}
Let $\tilde \lambda_{1, q}$ be the quantity on the right hand side in \eqref{equivalent lambda}. Observe that $\lambda_{1, q} \le \tilde \lambda_{1, q},$ because $W^{1,1}(\Omega) \subset BV(\Omega)$. Now, let $u \in BV(\Omega)\backslash\{0\}$ be such that $\int_\Omega |u|^{q-2}u\,d\mathcal{H}^N=0$. Using an approximation by mollification (see \cite[Theorem~1.17]{Giusti}), there is  a sequence $u_n \in C^\infty(\Omega) \cap W^{1,1}(\Omega)$ such that 
\[ 
\int_\Omega |D u_n|\,d\mathcal{H}^N \to |Du|(\Omega)
\quad \text{ and } \quad \int_\Omega |u_n- u|\,d\mathcal{H}^N \to 0. 
\]

Let $\tilde u_n:=u_n + c_q(u_n),$ where $c_q(u_n)$ is the constant in Lemma~\ref{lemma c pw}. Then, $\int_\Omega |\tilde u_n|^{q - 2} \tilde u_n\,d\mathcal{H}^N=0$ and, by continuity (see Lemma~\ref{lemma c pw}), $c_q(u_n) \to c_q(u)=0$.
Moreover, since $u_n$ is bounded in $W^{1,1}(\Omega)$, we have that,  up to a subsequence, $u_n\to u$ in $L^q(\Omega)$. 
Hence, 
\[ 
\frac{|Du|(\Omega)}{\norm{u}_q} = \lim_{n\to\infty} \frac{\int_\Omega |D\tilde u_n|\,d\mathcal{H}^N}{\norm{\tilde u_n}_q} \ge \tilde \lambda_{1, q}.
\]
Taking the infimum over all $u \in BV(\Omega)\backslash\{0\}$ such that $\int_\Omega |u|^{q-2} u\,d\mathcal{H}^N=0$, yields \eqref{equivalent lambda}.
\end{proof}

\begin{theo}\label{main theo}
Let  $q \in (1, N/(N-1))$. Then, any minimizer $u\in BV(\Omega)$ of $\lambda_{1, q}$ satisfies
\begin{align*}
- \operatorname{div}\left( \frac{D u}{|Du|} \right) = \lambda_{1, q} |u|^{q-2} u\quad \text{ in } \Omega, \qquad 
\partial_\nu \left( \frac{D u}{|Du|} \right)=0 \quad \text{ on } \partial \Omega,
\end{align*}
in the sense of Definition~\ref{def solution}.
\end{theo}

To prove this result, consider the problem
\[ \inf \{ E(u): \, \norm{u}_q=1\}, \]
where $E : L^{q}(\Omega) \to \overline \R$  is given by
\begin{align}\label{E:Def}
E(u):= \begin{cases}
|D u|(\Omega), & \text{ if } u \in BV(\Omega) \text{ and } \int_\Omega |u|^{q-2} u\,d\mathcal{H}^N= 0,\\
+\infty, & \text{ otherwise. }
\end{cases}
\end{align}
Notice that $E$ is not convex.

For $t>1$ and $f: L^t(\Omega) \to \overline \R$, the Legendre-Fenchel transform $f^*: L^{\frac{t}{t-1}}(\Omega) \to  \overline \R$ of $f$ is given by
\begin{align}\label{lf:def}
f^*(u^*):= \sup_{u \in L^t(\Omega)}\left \{ \int_\Omega u u^*\,d\mathcal{H}^N - f(u) \right \}. 
\end{align}
Moreover, the subdifferential of $E$ at $u \in L^q(\Omega)$ is 
\[
\partial E(u):= \left \{ g \in L^{\frac{q}{q-1}}(\Omega):\, E(\varphi) \ge E(u) +\int_\Omega g(\varphi-u)\,d\mathcal{H}^N\, \text{ for any } \varphi \in L^{q}(\Omega) \right\}. 
\]
Since $E$ is not convex, the subdifferential may be empty. Indeed, $\partial E(u) = \emptyset$ if $E(u) = + \infty$, see \cite[Equation (5.2) at p.21]{Ekeland}.  
The next proof is inspired by \cite[Proposition 4.23]{KS}.

\begin{prop}\label{subdiff}
    Let $u \in BV(\Omega)$ be such that $\int_\Omega |u|^{q-2} u\,d\mathcal{H}^N=0$ and let $f \in L^{\frac{q}{q-1}}(\Omega)$ be such that $\int_\Omega f\,d\mathcal{H}^N=0$. Then, $f \in \partial E(u)$ if and only if 
    $u$ is a solution to
    \begin{align*}
        -\operatorname{div}\left( \frac{Du}{|Du|} \right) = f\quad \text{ in }\Omega,\qquad
        \partial_\nu \left( \frac{Du}{|Du|} \right) =0 \quad\text{ on }\partial \Omega,    
    \end{align*}
in the sense of Definition~\ref{def solution}.    
\end{prop}
\begin{proof}
Let
\begin{align}\label{M:def}
M^*:=\left \{ v^* \in L^{\frac{q}{q-1}}(\Omega): \exists z \in L^\infty(\Omega, \R^N), \, \norm{z}_\infty \le 1, \, -\operatorname{div}z= v^*\text{ in }\Omega, \, \partial_\nu z=0\text{ on }\partial \Omega \right\},
\end{align}
where $-\operatorname{div}z= v^*\text{ in }\Omega$ and $\partial_\nu z=0\text{ on }\partial \Omega$ hold in the sense of Definition~\ref{def solution}, namely, \eqref{def:eq} holds with $f=v^*$.

In particular, any element $v^* \in M^*$ satisfies that $\int_\Omega v^* \varphi \,d\mathcal{H}^N= \int_\Omega z \cdot D \varphi\,d\mathcal{H}^N$ for all $\varphi \in C^\infty(\overline \Omega)$ and $\int_\Omega v^*\,d\mathcal{H}^N=0$. Let us show that $M^*$ is closed. Indeed, let $v_n^* \in M^*$ be such that $v_n^* \to v^*$ in $L^{\frac{q}{q-1}}(\Omega)$. Then there exist $z_n \in L^\infty(\Omega, \R^N)$ such that $\norm{z_n}_\infty \le 1$ and 
\begin{align}\label{zneqn}
\int_\Omega z_n \cdot D \varphi \,d\mathcal{H}^N= \int_\Omega v_n^* \varphi\,d\mathcal{H}^N \quad \text{ for all } \varphi \in C^\infty(\overline \Omega).    
\end{align}
Using that $L^\infty(\Omega;\mathbb{R}^N)
\cong
\left(L^1(\Omega;\mathbb{R}^N)\right)^*$ and the Banach-Alaoglu theorem, there exists $z \in L^\infty(\Omega, \R^N)$ such that, up to a subsequence, $z_n \rightharpoonup z$ weakly* in $L^\infty(\Omega, \R^N)$. Using this and \eqref{zneqn}, we obtain that $\int_\Omega z \cdot D\varphi\,d\mathcal{H}^N= \int_\Omega v^* \varphi\,d\mathcal{H}^N$ for all $\varphi \in C^\infty(\overline \Omega).$ This implies that $v^* \in M^*$ and therefore $M^*$ is closed.

Moreover, it is easy to see that $M^*$ is convex. Let 
\[ I_{M^*}(v^*):= \begin{cases}
0, & \text{ if } v^* \in M^*,\\
+\infty, & \text{ otherwise. }
    \end{cases} \]
Then, using \eqref{lf:def},
\[ 
I_{M^*}^*(v)
=\sup_{u \in L^q(\Omega)}\left \{ \int_\Omega u v\,d\mathcal{H}^N - I_{M^*}(u)
\right \}
= \sup_{v^* \in M^*} \int_\Omega v^* v\,d\mathcal{H}^N \quad \text{ with }v \in L^{q}(\Omega). 
\]

\underline{Step 1.1} Recall the definition of $E$ in \eqref{E:Def}. Next, we show that 
\begin{equation}\label{step 1.1} I_{M^*}^*(v) \le E(v) \quad \text{ for all } v \in L^{q}(\Omega). \end{equation}
Let $v \in  BV(\Omega)$ and $v^* \in M^*$. By definition of $M^*$,
\[
\int_\Omega v^* \psi\,d\mathcal{H}^N = - \int_\Omega \operatorname{div} z \cdot \psi \,d\mathcal{H}^N \quad \text{ for all } \psi \in C^\infty(\overline \Omega)\text{ and for some $z \in L^\infty(\Omega, \R^N)$.}
\]

By mollification (see \cite[Theorem~1.17]{Giusti}), there is $v_n \in C^\infty(\Omega) \cap W^{1,1}(\Omega)$ such that 
$\int_\Omega |D v_n| \,d\mathcal{H}^N \to |Dv|(\Omega)$ and $\int_\Omega |v_n-v|\,d\mathcal{H}^N \to 0$ as $n \to\infty$. Then,
\[ \int_\Omega v^* v\,d\mathcal{H}^N = \lim_{n\to\infty}\int_\Omega v^* v_n \,d\mathcal{H}^N= \lim_{n\to\infty}\int_\Omega z \cdot Dv_n \,d\mathcal{H}^N \le \lim_{n\to\infty}\int_\Omega |Dv_n|\,d\mathcal{H}^N = |Dv|(\Omega) \le E(v).  \]
Inequality \eqref{step 1.1} follows recalling that $E(v)=+\infty$ if $v \not \in BV(\Omega)$. 

\medskip

\underline{Step 1.2} We now prove that
\[
I_{M^*}^*(v) \ge E(v) \quad \text{ for all } v \in L^{q}(\Omega) \text{ such that } \int_\Omega |v|^{q-2}v\,d\mathcal{H}^N=0. 
\]

Let $v \in L^{q}(\Omega)$ be such that $\int_\Omega |v|^{q-2}v\,d\mathcal{H}^N=0$, then, by definition of $M^*$,
\begin{align*}
    I_{M^*}^*(v)&= \sup_{v^* \in M^*} \int_\Omega v^* v\,d\mathcal{H}^N\\
    &= \sup \left \{ -\int_\Omega v\,  \operatorname{div} z \,d\mathcal{H}^N:\,  z \in L^\infty(\Omega, \R^N), \,  \norm{z}_\infty \le 1, \, \operatorname{div}z \in L^{\frac{q}{q-1}}(\Omega), \, \partial_\nu z=0\text{ on }\partial \Omega \right \} \\
    & \ge \sup \left \{ - \int_\Omega v \, \operatorname{div} z\,d\mathcal{H}^N:\,  z \in C_0^\infty(\Omega, \R^N), \, \norm{z}_\infty \le 1 \right \}.
\end{align*}
By definition, the last term is the total variation of $v$. If $v \in BV(\Omega),$ then it is finite and coincides with $E(v)$ (recall that, by assumption, $\int_\Omega |v|^{q-2} v\,d\mathcal{H}^N=0$). If $v \in L^{q}(\Omega) \setminus BV(\Omega)$ then it is infinite, and again coincides with $E(v)$, showing the claim. 

Combining Steps 1.1 and 1.2 we get
\begin{equation}\label{step 1} I_{M^*}^*(v)=E(v) \quad \text{ for all } 
v \in L^{q}(\Omega) \text{ with } \int_\Omega |v|^{q-2}v\,d\mathcal{H}^N=0. \end{equation}

\underline{Step 2} Let $u^* \in L^{\frac{q}{q-1}}(\Omega)$ be such that $\int_\Omega u^*\,d\mathcal{H}^N=0$ and, for any $u \in L^q(\Omega)$, let $c_q(u)$ be the constant in Lemma~\ref{lemma c pw}. 
Since any $v^* \in M^*$ satisfies $\int_\Omega v^*=0$, one has that
\begin{align*}
    I_{M^*}^*(u+c_q(u))= \sup_{v^* \in M^*} \int_\Omega v^*(u+c_q(u))\,d\mathcal{H}^N=\sup_{v^* \in M^*} \int_\Omega v^* u\,d\mathcal{H}^N =I_{M^*}^*(u).
\end{align*}
Then,
\begin{equation}\label{dis I} 
\int_\Omega u u^*\,d\mathcal{H}^N - I_{M^*}^*(u)= \int_\Omega (u+c_q(u)) u^*\,d\mathcal{H}^N - I^*_{M^*}(u+c_q(u)). 
\end{equation} 

Since $E(u) =+\infty$ if $\int_\Omega |u|^{q-2}u\,d\mathcal{H}^N \ne 0$, we have that 
\begin{equation}\label{sup E} 
E^*(u^*)= \sup_{u \in L^q(\Omega)} \left\{ \int_\Omega u u^*\,d\mathcal{H}^N - E(u) \right \} = \sup_{\substack{u \in L^q(\Omega),\\ \int_\Omega |u|^{q-2}u=0}} \left\{ \int_\Omega u u^*\,d\mathcal{H}^N - E(u) \right \}. 
\end{equation}

Hence, using \eqref{step 1}, \eqref{dis I}, and \eqref{sup E}, for any $u^* \in L^{\frac{q}{q-1}}(\Omega)$ such that $\int_\Omega u^*\,d\mathcal{H}^N=0$ one has that
\begin{align*}
    E^*(u^*)&= \sup_{u \in L^q(\Omega)} \left\{ \int_\Omega u u^*\,d\mathcal{H}^N - E(u) \right \}
    = \sup_{\substack{u \in L^q(\Omega),\\ \int_\Omega |u|^{q-2}u\,d\mathcal{H}^N=0}} \left\{ \int_\Omega u u^*\,d\mathcal{H}^N - E(u) \right \}\\
    &= \sup_{\substack{u \in L^q(\Omega),\\ \int_\Omega |u|^{q-2}u\,d\mathcal{H}^N=0}} \left\{ \int_\Omega u u^*\,d\mathcal{H}^N - I_{M^*}^*(u) \right \}
    =  \sup_{u \in L^q(\Omega)} \left\{ \int_\Omega u u^*\,d\mathcal{H}^N - I_{M^*}^*(u) \right \}= (I_{M^*}^*)^*(u^*).
\end{align*}
Since $M^*$ is convex, the Fenchel–Moreau Theorem~implies that $(I_{M^*}^*)^*(u^*)=I_{M^*}(u^*)$ for any $u^* \in L^{\frac{q}{q-1}}(\Omega)$.  As a consequence, 
\begin{equation}\label{step 2} E^*(u^*) = I_{M^*}(u^*) \quad \text{ for any $u^* \in L^{\frac{q}{q-1}}(\Omega)$ such that $\int_\Omega u^*\,d\mathcal{H}^N=0$. } \end{equation}

\underline{Step 3} By \cite[Chapter 1, Proposition 5.1]{Ekeland}, given $u\in L^{q}(\Omega)$, $u^* \in L^{\frac{q}{q-1}}(\Omega)$ belongs to $\partial E(u)$ if and only if
\[ \int_\Omega u u^*\,d\mathcal{H}^N = E(u) + E^*(u^*). \]
By \eqref{step 2}, if $\int_\Omega u^*\,d\mathcal{H}^N =0$, then  $E(u) + E^*(u^*) = E(u) + I_{M^*}(u^*)$ and $u^* \in \partial E(u)$ if and only if $u^* \in M^*$ and $E(u)=\int_\Omega u u^*\,d\mathcal{H}^N$. This ends the proof, because $u^* \in \partial E(u)$ implies the existence of $z \in L^\infty(\Omega, \R^N)$ such that $\norm{z}_\infty \le 1$ and $\int_\Omega z \cdot D \varphi\,d\mathcal{H}^N = \int_\Omega u^* \varphi\,d\mathcal{H}^N$ for any $\varphi \in C^\infty(\overline \Omega).$ Moreover, if $E(u)=\int_\Omega u u^*\,d\mathcal{H}^N$, then $|Du|(\Omega)=\int_\Omega u u^*\,d\mathcal{H}^N$.
\end{proof}

\begin{prop}\label{min subdiff}
    Let $u \in BV(\Omega) $ be such that $\norm{u}_q=1$ and $E(u)=\min_{\norm{v}_q=1} E(v) = \lambda_{1, q}.$ Then there exists $f \in \partial E(u)$ such that 
\[ |u|^{q-2} u =\frac{1}{\lambda_{1, q}} f. \]
\end{prop}
\begin{proof}
By minimality of $u$, $E\left( \frac{v}{\norm{v}_q} \right) \ge E(u)$ for all $v \in L^q(\Omega)\backslash\{0\}$. Denote $G(v)=\norm{v}_q$ and take $g^* \in \partial G(u)$, namely $g^*=  |u|^{q-2} u.$ Then,
\[ G(v) \ge G(u) + \int_\Omega g^* (v-u)\,d\mathcal{H}^N=1+ \int_\Omega g^* (v-u)\,d\mathcal{H}^N \quad \text{ for any } v \in L^{q}(\Omega), \, v \ne 0. \]
Also, notice that $E\left( \frac{v}{\norm{v}_q} \right) = \frac{1}{\norm{v}_q} E(v).$ Therefore,
\[
    E(v) \ge G(v) E(u) \ge E(u) + E(u) \int_\Omega g^* (v-u)\,d\mathcal{H}^N \quad \text{ for all } v \in L^q(\Omega),\, v \ne 0, 
\]
which implies that $E(u) g^* \in \partial E(u)$. As a consequence, there exists $f \in \partial E(u)$ such that $|u|^{q-2} u = \frac{1}{E(u)} f.$ This implies the claim with $t= (E(u))^{-1}= \frac{1}{\lambda_{1,q}}.$
\end{proof}

\begin{proof}[Proof of Theorem~\ref{main theo}]
Let $u \in BV(\Omega)$ be such that $\int_\Omega |u|^{q-2} u \,d\mathcal{H}^N =0$ and $\norm{u}_q=1$ which attains 
\[ \lambda_{1, q}=E(u)=\min_{\norm{v}_q=1} E(v). \]
Recall that a minimizer for $\lambda_{1, q}$ exists due to Lemma~\ref{att}. 
Then, by Proposition~\ref{min subdiff}, there is $f \in \partial E(u)$ such that  $|u|^{q-2}u= \frac{1}{\lambda_{1, q}} f$. Moreover, since $\int_\Omega |u|^{q-2} u \,d\mathcal{H}^N =0$, one has that $\int_\Omega f\,d\mathcal{H}^N=0$. Then $f$ satisfies the hypotheses of Proposition~\ref{subdiff} and therefore  $u$ is a solution of 
\[
-\operatorname{div}\left( \frac{Du}{|Du|} \right) = \lambda_{1, q} |u|^{q-2} u \quad \text{ in } \Omega,\qquad 
    \partial_\nu \left( \frac{Du}{|Du|} \right) = 0 \quad \text{ on } \partial \Omega. \qedhere
\]
\end{proof}

\subsubsection{Case \texorpdfstring{$q=1$}{q=1}}
In the case $q=1$, the corresponding formal eigenvalue problem is 
\[
-\operatorname{div} \left( \frac{D u}{|Du|} \right) = \lambda \, \sgn(u) \qquad  \text{ in } \Omega. 
\]
To give this formula a rigorous interpretation, let us define the set-valued sign function 
\[ \sgn(t):=\begin{cases}
    1, & \text{ if } t >0, \\
    [-1, 1], & \text{ if } t=0,\\
    -1, & \text{ if } t <0. 
\end{cases} \]

If $u:\Omega\to\R$ is a function, then $s\in \sgn(u)$ is a function $s:\Omega\to\R$ such that $s(x)=1$ if $u(x)>0$, $s(x)=-1$ if $u(x)<0$, and $s(x)\in[-1,1]$ if $u(x)=0$.  We call $s$ a selection of $\sgn(u).$

In the Dirichlet case (see \cite{KS}),
\[ 
\lambda_D := \inf\left\{ |Du|(\Omega) + \int_{\partial \Omega}|u|: u \in BV(\Omega), \norm{u}_1=1 \right\} 
\]
is attained, and for any minimizer $u \in BV(\Omega)$ there exists $z \in L^\infty(\Omega, \R^N)$ and $s \in \sgn(u)$ such that 
\[ \norm{z}_{L^\infty} \le 1, \quad -\int_\Omega \operatorname{div}(z) u \,d\mathcal{H}^N= |Du|(\Omega) \]
and 
\begin{equation}\label{limit eq dir} -\operatorname{div}(z)=\lambda_D \, s \quad \text{a.e. on }\Omega.  \end{equation}

Actually, in \cite{KS} it is shown that, for any selection $s \in \sgn(u),$ there exists a solution in the above sense. As expected, there are solutions to \eqref{limit eq dir} which are not minimizers for $\lambda_D$, see \cite{Chang}. 
We also recall that $\lambda_D$ and the corresponding eigenfunction are related to the Cheeger constant, see \cite{KF}.

In \cite{Chang}, a sequence of eigenvalues of $\Delta_1$ is found via Ljusternik–Schnirelmann theory for nonsmooth functionals, where critical points are to be interpreted as functions with zero slope and this approach also applies to the Neumann case. Here, we give a different variational characterization with respect to \cite{Chang}, inspired by \cite{PW}. Apart from its intrinsic interest, this approach allows us to avoid the use of weak/strong slopes, and it turns out to be more suitable to prove the convergence results in Section~\ref{sec convergence}, as well as to provide a  geometrical characterization of $\lambda_{1,1}$, see Section~\ref{geom char}. 
We also mention that in \cite{gg,g01} a different formulation of $\lambda_{1,1}$ is given as
\[ \inf\left\{ |Du|(\Omega): \, u \in BV(\Omega) \setminus \{ 0 \}, \, \int_\Omega \sgn(u)\,d\mathcal{H}^N=0 \right\} .\] 
However, we show in Remark \ref{ggremark} that this infimum is, in general, not attained.

We give the following definition of eigenfunctions and eigenvalues of $\Delta_1$ with Neumann conditions. 
\begin{defi}\label{def eigenpair}
    We say that $(\lambda, u) \in \R \times (BV(\Omega) \setminus \{ 0 \})$ is an eigenpair for $\Delta_1$ with Neumann boundary conditions if there exists a selection $s \in \sgn(u)$ with $\int_\Omega s=0$ and 
   such that 
$u$ is a solution of 
\begin{align}\label{selection:prob}
        -\operatorname{div} \left( \frac{Du}{|Du|} \right) = \lambda s \quad \text{ in }\Omega,
        \qquad 
    \partial_\nu u=0 \quad \text{ on }\partial \Omega,
\end{align}
  in the sense of Definition~\ref{def solution}.  
\end{defi}
This definition also appears in \cite{SdLSdL, SdLSdL2}. However, in \cite{SdLSdL, SdLSdL2} no variational characterization of the first eigenvalue is provided.

Recall the definition of $\mathfrak{N}(\Omega)$ given in \eqref{N omega} and let
\begin{align}
\label{lambda 1 1} \lambda_{1, 1}:= \inf \left\{ \frac{|Du|(\Omega)}{\norm{u}_1}: u \in \mathfrak{N}(\Omega) \setminus \{0\} \right\} 
 = \inf \left\{|Du|(\Omega): u \in \mathfrak{N}(\Omega), \; \norm{u}_1=1 \right\}. \end{align}

The following is the analogue of Lemma~\ref{lambda1qlem} in the case $q=1.$

\begin{lemma}
One has
\begin{equation}\label{equivalent lambda 1}
\lambda_{1, 1}= \inf \left\{|D u|(\Omega): \, u \in W^{1,1}(\Omega), \, ||\{ u >0 \}|- |\{ u < 0 \}|| \le |\{u=0\}|, \,\norm{u}_1=1 \right\}. 
\end{equation}
\end{lemma}
\begin{proof}
Let $\tilde \lambda_{1, 1}$ denote the right hand side in \eqref{equivalent lambda 1}. 
Since $W^{1,1}(\Omega) \subset BV(\Omega)$ then $\lambda_{1, 1} \le \tilde \lambda_{1, 1}$. Let now $u \in \mathfrak {N}(\Omega)$ be such that $u\ne 0$. By \cite[Theorem~1.17]{Giusti}, there exists a sequence $u_n \in C^\infty(\Omega) \cap  W^{1, 1}(\Omega)$ such that 
\[ 
\int_\Omega |D u_n|\,d\mathcal{H}^N \to |Du|(\Omega), \quad \int_\Omega |u_n- u|\,d\mathcal{H}^N \to 0. 
\]
Using the notation in Lemma~\ref{c BV}, let $\tilde u_n:=u_n + \underline c(u_n)\in \mathfrak {N}(\Omega)$. By Lemma~\ref{c BV}, $\liminf\limits_{n\to \infty} \underline c(u_n) \ge \underline c(u)$. Hence, by Fatou's Lemma and Remark~\ref{rmk equality c}, 
\[ \frac{|Du|(\Omega)}{\norm{u}_1}= \frac{|Du|(\Omega)}{\norm{u+ \underline c (u)}_1} \ge \lim_{n\to\infty}\frac{\int_\Omega |D\tilde u_n|\,d\mathcal{H}^N}{\norm{\tilde u_n}_1} \ge \tilde \lambda_{1, 1}. \]
Taking the infimum over $u\in \mathfrak{N}(\Omega) \setminus \{0\}$ yields \eqref{equivalent lambda 1}.
\end{proof}
\begin{prop}\label{lambda 11 att}
One has that $\lambda_{1,1}$ is attained and $\lambda_{1,1} >0$. 
\end{prop}
\begin{proof}
Let $u_n \in \mathfrak {N}(\Omega)$ be a minimizing sequence for $\lambda_{1, 1}$ such that $\norm{u_n}_1=1$. By \eqref{comp:em}, there exists a function $u \in BV(\Omega)$ such that $u_n \to u$ in $L^1(\Omega)$.  In particular, $\norm{u}_1=1$. 
We now prove that  $u \in \mathfrak {N}(\Omega)$. 
By \eqref{charact M}, 
\[ \int_\Omega \sgn_-(u_n)\,d\mathcal{H}^N \le 0 \le \int_\Omega \sgn_+(u_n) \,d\mathcal{H}^N\]
and, passing to the limit, $\int_\Omega \sgn_-(u)\,d\mathcal{H}^N \le 0 \le \int_\Omega \sgn_+(u)\,d\mathcal{H}^N,$ which implies that $u \in \mathfrak {N}(\Omega)$. This shows that $\lambda_{1,1}$ is attained.

To see that $\lambda_{1,1} >0$, let
\[ t(u):=\sup \{ t\in \R\::\: |\Omega_t| \ge |\Omega \setminus \Omega_t|\}, \quad \text{ where } \Omega_t:=\{x \in \Omega: \, u(x) >t\}. \]
Let us show that $0\leq  t(u) \le C$. Using that $u \in \mathfrak {N}(\Omega)$, for any $t >0$,
\begin{equation}\label{dis t(u)} 
|\Omega_t|=|\{u > t\}| = |\{u>0\}|-|\{0 < u \le t\}| \le |\{u \le 0\}| + |\{0 < u \le t\}| = |\{ u \le t \}|=|\Omega\backslash\Omega_t|. 
\end{equation}
Moreover, the inequality is strict if $|\{0 < u \le t\}| \ne 0$. Therefore, if $t(u) = +\infty$, then there exists a sequence $t_n \to + \infty$ such that $|\{0 < u \le t_n\}|=0$ for any $t_n$. However, in this case, $\{u\le 0\}=\Omega$, and $t(u) \le 0$, a contradiction. 
On the other hand, we notice that, for any $t<0$,
\[ 
|\Omega\backslash\Omega_t|=|\{u \le t\}| = |\{ u <0 \}|- |\{ t < u < 0\}|\le |\{u \ge 0\}| + |\{ t < u < 0\}|= |\{u > t\}|=|\Omega_t|, 
\]
thus $t(u) \ge 0$.

With this definition of $t(u)$, \cite[Theorem~II]{flemingrishel} implies that $\norm{u-t(u)}_{\frac{N}{N-1}} \le C |Du|(\Omega)$ for some $C>0$.  Then, using Lemma~\ref{c BV},
\[ 
1=\norm{u}_1\le \norm{u-t(u)}_1 \le C_1 \norm{u-t(u)}_{1^*} \le C_2 |Du|(\Omega)
\]
for some $C_1, C_2>0$. This implies that $\lambda_{1,1 }>0$ and ends the proof. 
\end{proof}

We now follow the same ideas of Theorem~\ref{main theo}. 
We define the functional $\tilde E : L^{1}(\Omega) \to \overline \R$  such that 
\begin{align}\label{tildeEdef}
\tilde E(u):= \begin{cases}|D u|(\Omega), & \text{ if } u \in \mathfrak {N}(\Omega),\\
+\infty, & \text{ otherwise. }
\end{cases}
\end{align}
\begin{prop}\label{subdiff q1}
    Let $u \in \mathfrak {N}(\Omega)$. Then, $f \in L^\infty(\Omega)$ such that $\int_\Omega f\,d\mathcal{H}^N=0$ satisfies $f \in \partial \tilde E(u)$ if and only if
    $u$ is a solution to 
    \begin{align*}
        - \operatorname{div} \left( \frac{Du}{|Du|} \right) = f \quad  \text{ in }\Omega, \qquad
        \partial_\nu \left( \frac{Du}{|Du|} \right) =0 \quad \text{ on }\partial \Omega,
    \end{align*}
    in the sense of Definition~\ref{def solution}.
\end{prop}
\begin{proof}
Let $M^*$ be given by \eqref{M:def}. Reasoning as in Step 1 of the proof of Proposition~\ref{subdiff}, in place of \eqref{step 1} we now have that
\[ I_{M^*}^*(v)=\tilde E(v) \quad \text{ for all } 
v \in L^{1}(\Omega), \; \abs{ |\{ v >0\}|-|\{ v <0\}|} \le |\{v=0\}|. \]
In Step 2 we used the constant $c_q(u)$ given by Lemma~\ref{lemma c pw}. Here, instead we use any $c(u)$ such that $u+ c(u) \in \mathfrak {N}(\Omega)$ (provided by Lemma~\ref{c BV}), to obtain \eqref{step 2}. 
Step 3 follows in the same way, and the proof is complete. 
\end{proof}
The proof of the next Proposition is completely analogous to Proposition~\ref{min subdiff}, we omit the proof. 
\begin{prop}\label{min subdiff q1}
    Let $u \in BV(\Omega) $ be such that $\tilde E(u)=\min_{\norm{v}_1=1} \tilde E(v) \in (0, \infty).$  Let $G(v):=\norm{v}_1$. 
Then, for any $g^* \in \partial G(u)$, there exists $f \in \partial \tilde E(u)$ such that 
\[ g^* =\frac{1}{\lambda_{1,1}} f. \]
\end{prop}

\begin{lemma}\label{lemma g*}
    Let $G(u):=\norm{u}_1$ and $u \in \mathfrak {N}(\Omega)$. There exists $\alpha^* \in [-1, 1]$ such that 
    \[ 
    g^*(x):= \begin{cases}
        1, & \text{ if } u(x) >0, \\
        \alpha^*, & \text{ if } u(x) =0, \\
        -1, & \text{ if } u(x) <0,
    \end{cases}
    \]
   satisfies 
    $g^* \in \partial G(u)$ and $\int_\Omega g^*\,d\mathcal{H}^N =0$.
\end{lemma}
\begin{proof}
Recall that  
\begin{equation}\label{subgradient} \partial (|y_0|)= \begin{cases}
    1, & \text{ if } y_0 >0, \\
    [-1, 1], & \text{ if } y_0 =0, \\
    -1 & \text{ if } y_0 <0, 
\end{cases}\end{equation}
and, by definition of subdifferential, for any $y_0 \in \R$, 
\begin{equation}\label{dis subg} 
\ell_{y_0} \in \partial (|y_0|)
\quad \text{ if and only if }\quad 
|y| -|y_0| \ge \ell_{y_0} (y-y_0) \quad \text{ for any $y\in \R $}.
\end{equation}
For any $\alpha \in [-1, 1]$ and $x \in \Omega$, let us define 
\[ g_\alpha(x):= \begin{cases}
1, & \text{ if } u(x)>0,\\
\alpha,  & \text{ if } u(x)= 0,\\
- 1, & \text{ if } u(x)<0. 
\end{cases}\]
Then, by \eqref{subgradient}, for any fixed $x \in \Omega$, $g_\alpha(x)$ belongs to $\partial (|y_0|)$, with $y_0=u(x)$. In particular, for any fixed $x \in \Omega$, and for any $v \ne 0$, consider $y_0:=u(x)$ and $y:=v(x)$. Then \eqref{dis subg} gives 
\[ 
|v(x)| -|u(x)| \ge g_{\alpha} (v(x)-u(x)).  
\] 
Integrating,
\[ \int_\Omega |v(x)|\,d\mathcal{H}^N \ge \int_\Omega |u(x)|\,d\mathcal{H}^N + \int_\Omega g_\alpha(x) (v(x)-u(x)) \,d\mathcal{H}^N\quad \text{ for all $v \in L^1(\Omega)$, \, $v \ne 0$}. \]
Namely, $g_\alpha \in \partial G(u)$.  Now, observe that 
\[ \int_\Omega g_\alpha \,d\mathcal{H}^N= |\{ u>0\}| - |\{ u <0 \}| + \alpha |\{ u =0 \}|. \]
Let us distinguish two situations. 
If $|\{ u=0\}| \ne 0$, then we choose $g^*= g_{\alpha^*}$ with 
\[
\alpha^*:= \frac{|\{ u<0\}| - |\{ u >0 \}|}{|\{ u =0 \}|}, 
\]
and observe that $\alpha^*\in[-1, 1]$, because $u \in \mathfrak {N}(\Omega)$. 
On the other hand, if $|\{ u=0\}| = 0$,  then we choose $g^*=\sgn(u)$, and notice that
\[ \abs{\int_\Omega \sgn(u)\,d\mathcal{H}^N}= \abs{|\{ u>0\}| - |\{ u<0\}|} \le |\{u=0\}| =0, \]
again because $u \in \mathfrak {N}(\Omega)$.
\end{proof}

\begin{theo}\label{main theo 2}
$\lambda_{1,1 } $ is attained. If $u \in BV(\Omega)$ is a minimizer of $\lambda_{1,1}$, then $(\lambda_{1,1}, u)$ is an eigenpair of $\Delta_1$ in the sense of Definition~\ref{def eigenpair} with 
\[ s(x):=\begin{cases}
 1, & \text{ if } u(x) >0, \\
        \alpha^*, & \text{ if } u(x) =0, \\
        -1, & \text{ if } u(x) <0, 
\end{cases}
\]
where 
\[ \alpha^*= \begin{cases}
    0, & \text{ if } |\{u=0\}|=0, \\
    \frac{|\{ u<0\}| - |\{ u >0 \}|}{|\{ u =0 \}|}, & \text{ otherwise. }
\end{cases}\]
Moreover, $\lambda_{1,1}$ is the smallest positive value such that there exists an eigenfunction of $\Delta_1$ in the sense of Definition~\ref{def eigenpair}. 
\end{theo}
\begin{proof}
By Proposition~\ref{lambda 11 att}, we know that there exists $u \in \mathfrak {N}(\Omega)$ such that $\norm{u}_1=1$, which attains $\lambda_{1, 1}$. Then, by \eqref{tildeEdef},
\[ 
\tilde E(u)=\min\{\tilde E(v)\::\: v\in \mathfrak {N}(\Omega),\ \norm{v}_1=1  \} \in (0, +\infty). 
\]
Let $G(u):=\norm{u}_1$. By Proposition~\ref{min subdiff q1}, for any $g^* \in \partial G(u)$ there exists $f \in \partial \tilde E(u)$ such that  $g^*= \frac{1}{\lambda_{1,1}} f$. 
We choose, by Lemma~\ref{lemma g*}, $g^* \in \partial G(u)$ such that $\int_\Omega g^*\,d\mathcal{H}^N=0$, thus $\int_\Omega f\,d\mathcal{H}^N=0$.
Therefore, $f$ satisfies the hypotheses of Proposition~\ref{subdiff q1}, which gives the first conclusion.

Next we prove that $\lambda_{1,1}$ is the smallest eigenvalue of $\Delta_1$.  Let $\lambda>0$ be such that there is $v \in BV(\Omega) \setminus \{ 0 \}$ and a selection $s \in \sgn(v)$ such that 
\begin{align*}
    -\operatorname{div} \left( \frac{Dv}{|Dv|} \right)= \lambda s \quad\text{ in } \Omega,\qquad 
    \partial_\nu v=0 \quad \text{ on } \partial \Omega, 
\end{align*}
then $|Dv|(\Omega)=\lambda \int_\Omega s v\,d\mathcal{H}^N=\lambda \int_\Omega |v|\,d\mathcal{H}^N.$ Moreover, $0=\int_\Omega s\,d\mathcal{H}^N= |\{ v >0 \}| - |\{ v <0 \}| + \int_{\{ v=0 \}} s \,d\mathcal{H}^N$. Thus, $||\{ v >0 \}| - |\{ v <0 \}|| \le  |\{ v=0\}|,$ and, by definition,
\[
\lambda_{1,1} \le \frac{|Dv|(\Omega)}{\int_\Omega |v|\,d\mathcal{H}^N}=\lambda. \qedhere 
\]
\end{proof}

\begin{remark}
    In Section~\ref{counterex} we show an example in which any minimizer of $\lambda_{1,1}$ is such that $|\{ u=0\}| \ne 0$ and $\int_\Omega \sgn(u)\,d\mathcal{H}^N\ne 0$, in particular, $s \ne \sign(u)$ in general. 
\end{remark}

We are ready to prove Theorem~\ref{theo p=1 intro}.

\begin{proof}[Proof of Theorem~\ref{theo p=1 intro}]
Claim $(i)$ follows from Lemma~\ref{att} and Theorem~\ref{main theo}. Claim $(ii)$ follows from Theorem~\ref{main theo 2}.
\end{proof}

\section{Asymptotic analysis}\label{sec convergence}

Recall the definitions of $\lambda_{p, q}$ in \eqref{lambda pq def}, \eqref{lambda 1 q}, and \eqref{lambda 1 1}.  In this section, we prove Theorem~\ref{theo conv intro}.

We begin by showing some uniform bounds. Recall the constant $c_q(\cdot)$ from Lemma~\ref{lemma c pw}. 
\begin{lemma}\label{lem bdd seq}
Using the notation of Theorem~\ref{theo conv intro}, let $u_n$ be a minimizer of $\lambda_n:=\lambda_{p_n, q_n}$ with $\norm{u_n}_{q_n}=1$. Then, up to a subsequence, there exists a positive constant $C$ such that
\begin{equation}\label{unif bounds} \norm{D u_n}_{p_n} \le C \quad \text{ if } p_n >1, \qquad |D u_n|(\Omega) \le C \quad \text{ if } p_n=1\qquad \text{ for all }n\in \mathbb N.
\end{equation}
\end{lemma}
\begin{proof}
\underline{\textit{Step 1.} Case $q>1$:}  Let $\psi \in C^\infty(\overline \Omega)$ be such that $\int_\Omega |\psi|^{q-2}\psi\,d\mathcal{H}^N=0$; in particular, $c_q(\psi)=0$. We can assume that $q_n >1$ passing to a subsequence. 
Let $c_n:=c_{q_n} (\psi)$. Note that $c_n \le C$ independent of $n$. Indeed, if $c_n \to \pm \infty$ as $n \to \infty$, then
$0= \int_\Omega |\psi+ c_n|^{q_n-2} (\psi+ c_n)\,d\mathcal{H}^N \to \pm \infty,$ a contradiction. Thus, up to a subsequence, $c_n \to k$ for some $k$. Therefore, by Lemma~\ref{lemma c pw}, 
\begin{equation}\label{limit norm}
\norm{\psi}_q= \lim_{n\to\infty}\norm{\psi}_{q_n} \ge \lim_{n\to\infty}\norm{\psi + c_n}_{q_n} = \norm{\psi + k}_q \ge \norm{\psi}_q. 
\end{equation}

Then, if $p_n >1$, 
\[ \norm{D u_n}_{p_n}= \lambda_n \le \frac{\norm{D (\psi + c_n)}_{p_n}}{\norm{\psi + c_n}_{q_n}}  \to \frac{\norm{D \psi}_p }{\norm{\psi}_q}.  \]
Similar arguments show the bound in \eqref{unif bounds} if $p_n=1$. 

\medskip

\underline{\textit{Step 2.} Case $q=1$:} Let $\psi \in C^\infty(\overline \Omega) \cap \mathfrak {N}(\Omega)$ and assume that $p_n >1$ (the case $p_n=1$ is similar). Up to a subsequence, we can assume $q_n >1$ for any $n$, or $q_n \equiv 1$. Let us consider the first case, and let $c_n:=c_{q_n}(\psi)$. As in Step 1, we observe that $c_n$ is uniformly bounded, thus $c_n \to k$ up to a subsequence. Moreover, by Lemma~\ref{lemma c pw} and Lemma~\ref{c BV}, \eqref{limit norm} holds with $q=1$. 
Then, 
\[ \norm{D u_n}_{p_n} = \lambda_n \le \frac{\norm{D(\psi + c_n)}_{p_n}}{\norm{\psi + c_n}_{q_n}} \to \frac{\norm{D \psi}_p}{\norm{\psi}_1}. \]
The case $q_n \equiv 1$ is easier and it follows by taking $\psi$ as a competitor for $\lambda_n$. 
\end{proof}

We are now ready to prove Theorem~\ref{theo conv intro}.
\begin{proof}[Proof of Theorem~\ref{theo conv intro}]
We follow the proof of \cite[Proposition 2.22]{AST}.  Let $p_n, p \in [1, N)$, $q_n \in [1, p_nN/(N-p_n))$, and  $q \in [1, pN/(N-p))$ such that $(p_n,q_n)\to (p,q)$ as $n\to\infty$. Let $u_n$ be a minimizer for $\lambda_{p_n, q_n}$. 

\medskip

\underline{\textit{Step 1.} Case $p=1$. }     Let $p_n >1$ for any $n$, thus $u_n \in W^{1, p_n}(\Omega)$. The case $p_n \equiv 1$ can be treated similarly. By \eqref{unif bounds},
    \begin{equation} \label{estimate norm p_n} 
    \norm{D u_n}_1 \le |\Omega|^{\frac{p_n-1}{p_n}} \norm{D u_n}_{p_n} \le C;
    \end{equation}
    hence, $u_n$ is uniformly bounded in $BV(\Omega)$. By \eqref{comp:em}, passing to a subsequence, there exists $u \in BV(\Omega)$ such that $u_n \to u$ strongly in $L^r(\Omega)$ for any $r \in [1, N/(N-1))$.
    Since $p=1$, then $q < N/(N-1)$. Choosing $\bar t:=1/2(q+N/(N-1))$ we conclude that $u_n \to u$ strongly in $L^{\bar t}(\Omega)$, and there exists $h \in L^{\bar t}(\Omega)$ such that, up to a subsequence, $|u_n| \le h$ a.e. in $\Omega$ (see \cite[Lemma A.1]{willem}). Then, if $n$ is large enough so that $q_n < \bar t$, 
    by dominated convergence, 
        $1= \norm{u_n}_{q_n} \to \norm{u}_q$ and $ u \ne 0$. 
    Moreover,    recalling \eqref{estimate norm p_n}, 
    \[
    |Du|(\Omega) \le \liminf_{n\to\infty}\int_\Omega |D  u_n|\,d\mathcal{H}^N \le \liminf_{n\to\infty}\norm{D u_n}_{p_n}. 
    \]
We now distinguish two subcases.

\medskip

\underline{\textit{Step 1.1.} Case $q>1$:} In this case, we can assume without loss of generality that $q_n >1$ as well. By Lemma~\ref{lemma c pw}  and since $\norm{u_n}_{q_n} \to \norm{u}_q$, 
\begin{align}\label{pnqneq10}
\norm{u + c_q(u)}_q = \lim_{n\to\infty}\norm{u_n + c_q(u)}_{q_n} \ge \lim_{n\to\infty}\norm{u_n}_{q_n} = \norm{u}_q \ge \norm{u + c_q(u)}_q. 
\end{align}
 Then,
 \begin{align}\label{pnqneq1}
 \lambda_{1, q} \le \frac{|D(u+c_q(u))|(\Omega)}{\norm{u+ c_q(u)}_q} = \frac{|Du|(\Omega)}{\norm{u+ c_q(u)}_q} \le \liminf_{n\to\infty}\frac{\norm{D u_n}_{p_n}}{\norm{u_n}_{q_n}}=  \liminf_{n\to\infty}\lambda_{p_n, q_n}. 
  \end{align} 

By \eqref{equivalent lambda}, for any fixed $\varepsilon>0$ there exists $v \in W^{1,1}(\Omega) \setminus \{0\}$ such that $\int_\Omega |v|^{q-2}v\,d\mathcal{H}^N=0$ and 
\[ \frac{\int_\Omega |D v|\,d\mathcal{H}^N}{\norm{v}_q} \le \lambda_{1, q} + \varepsilon. \]
We can also find a sequence $v_j \in C^1(\overline \Omega) \subset W^{1, p_n}(\Omega)$ such that $D v_j \to D v $ in $L^1(\Omega)$ and $v_j \to v$ in $L^q(\Omega)$.
Moreover, for any fixed $j$, $c_{q_n}(v_j)$ is bounded (see the proof of Lemma~\ref{lem bdd seq}) and converging up to a subsequence to some value $k$ as $n \to \infty$, thus arguing as in \eqref{limit norm},
\[ \lim_{n\to\infty}\norm{v_j + c_{q_n}(v_j)}_{q_n} = \norm{v_j + k}_q. \]
Moreover, 
\[ \norm{v_j + c_q(v_j)}_q= \lim_{n\to\infty}\norm{v_j + c_q(v_j)}_{q_n} \ge \lim_{n\to\infty}\norm{v_j + c_{q_n}(v_j)}_{q_n} = \norm{v_j + k}_q \ge \norm{v_j + c_q(v_j)}_q.  \]
Hence, by dominated convergence and using the continuity of $c_q(\cdot)$, 
\[ \lambda_{1, q} + \varepsilon \ge \frac{\int_\Omega |D v|\,d\mathcal{H}^N}{\norm{v}_q} = \lim_{j\to\infty} \frac{\int_\Omega |D v_j|\,d\mathcal{H}^N}{\norm{ v_j + c_q(v_j)}_q} = \lim_{j\to\infty} \lim_{n\to\infty}\frac{\norm{D v_j}_{p_n}}{\norm{ v_j + c_{q_n}(v_j)}_{q_n}} \ge \lim_{n\to\infty}\lambda_{p_n, q_n}.  \]

Letting $\varepsilon \to 0$ we obtain, together with \eqref{pnqneq1}, that \eqref{pnqnclaim} holds, namely, that $\lambda_{p_n,q_n} \to \lambda_{1,q}$ as $(p_n, q_n) \to (1, q)$. Now, by \eqref{pnqneq10} and \eqref{pnqneq1}, we obtain that $\lim_{n\to\infty}\norm{D u_n}_{p_n}=|Du|(\Omega)$ and $\lim_{n\to\infty}\norm{u_n}_{q_n} = \norm{u}_q$. Now, \eqref{pnqnclaim2} follows by \eqref{comp:em}.

\medskip

\underline{\textit{Step 1.2.} Case $q=1$:} We can assume $q_n >1$, the case $q_n \equiv 1$ is similar. As in Step 1.1, up to substituting $c_q(u)$ with any constant $c$ such that $u+c \in \mathfrak {N}(\Omega)$ (see Lemma~\ref{c BV}), we have that
\[ \lambda_{1, 1} \le \liminf_{n\to\infty}\lambda_{p_n, q_n}. \]

By \eqref{equivalent lambda 1}, there is $v \in W^{1,1}(\Omega) \setminus \{ 0 \}$ such that $v \in \mathfrak {N}(\Omega)$ and
\[ \frac{\int_\Omega |D v|\,d\mathcal{H}^N}{\norm{v}_1} \le \lambda_{1,1} + \varepsilon. \]
We find a sequence $v_j \in C^1(\overline \Omega) \subset W^{1, p_n}(\Omega)$ such that $D v_j \to D v $ in $L^1(\Omega)$ and $v_j \to v$ in $L^1(\Omega)$.  Hence, again by similar arguments as for \eqref{limit norm}, if $c_{q_n} \to k$, 
\[ \norm{v_j + \underline c(v_j)}_1 \le \norm{v_j + k}_1 = \lim_{n \to\infty} \norm{v_j + c_{q_n}(v_j)}_{q_n}. \]

Recalling Remark~\ref{rmk equality c} and using that $\liminf\limits_{j\to\infty} \underline c (v_j) \ge \underline c(v)$ (by Lemma~\ref{c BV}),
\[\lambda_{1,1} + \varepsilon \ge \frac{\int_\Omega |D v|\,d\mathcal{H}^N}{\norm{v}_1} = \frac{\int_\Omega |D v|\,d\mathcal{H}^N}{\norm{v+ \underline c(v)}_1} \ge  \lim_{j\to\infty} \frac{\int_\Omega |D v_j|\,d\mathcal{H}^N}{\norm{v_j + \underline c(v_j)}_1} \ge \lim_{j\to\infty} \lim_{n\to\infty}\frac{\norm{D v_j}_{p_n}}{ \norm{v_j + c_{q_n}(v_j)}_{q_n}} \ge \lambda_{p_n, q_n}, \]
and we conclude as in Step 1.1.
\medskip

\underline{\textit{Step 2.} Case $p>1$:} For any $\varepsilon>0$,
\begin{equation}\label{bound p 1} \norm{Du_n}_{p - \varepsilon} \le |\Omega|^{\frac{p_n - p +\varepsilon}{p_n}} \norm{Du_n}_{p_n} \le C \quad \text{ with $n$ large enough. } \end{equation}
Thus there exists $u_\varepsilon$ such that $u_n \rightharpoonup u_\varepsilon$ in $W^{1, p- \varepsilon}(\Omega)$. Notice that since $u_n \to u_\varepsilon$ a.e. for every $\varepsilon$, then actually $u_\varepsilon=u$ is independent of $\varepsilon$. 
This allows one to pass to the limit as $\varepsilon \to 0$ in \eqref{bound p 1}, and get
\[ \norm{Du_n}_p \le |\Omega|^{\frac{p_n- p}{p_n} }\norm{Du_n}_{p_n} \le C \quad \text{ for $n$ large enough } \]
and
\[ \norm{Du}_p \le \liminf_{n\to\infty} \norm{Du_n}_p \le \liminf_{n\to\infty}\norm{Du_n}_{p_n}.  \]
The proof now proceeds exactly as in the case $p=1$ (Steps 1.1 and 1.2), up to substituting the total variation with the $L^p$ norm of the gradient.  We omit the details.
\end{proof}

\section{Geometric characterization and isoperimetric structure}\label{geom char}

As before, let $\Omega\subset \mathbb R^N$ be a bounded Lipschitz domain. Recall the condition \eqref{Nomega} and the equivalent condition \eqref{rem} stated in the introduction.  As defined in the introduction, we have that $P(\omega; \Omega) := |D \chi_\omega|(\Omega)$ for $\omega \subset \Omega$ and
\begin{align*}
\widetilde h_{\mathcal N}(\Omega)&:= \inf\left\{\frac{P(\omega_1; \Omega) + P(\omega_2; \Omega)}{|\omega_1 \cup \omega_2|}\::\: \omega_1,\omega_2\subset \Omega\text{ satisfy \eqref{Nomega}}\right\},\\
h_{\mathcal N}(\Omega)&:= \inf \left\{ \frac{P(\omega; \Omega)}{|\omega|}:\, \omega\subset \Omega, \ 0 <|\omega| \le \frac{|\Omega|}{2} \right\}.
\end{align*}

The following result shows that $\widetilde h_{\mathcal N}(\Omega)$ and $h_{\mathcal N}(\Omega)$ are equal.
\begin{lemma}\label{lemma charac2}
 If $(D^+,D^-)$ is a minimizer for $\widetilde h_{\mathcal N}(\Omega)$ such that $D^+, D^- \ne \emptyset$, then  $(D^+, \emptyset)$ and $(D^-, \emptyset)$ are also minimizers for $\widetilde h_{\mathcal N}(\Omega)$ and
  \begin{equation}\label{charact 2}
  \widetilde h_{\mathcal N}(\Omega) = 
  h_{\mathcal N}(\Omega)
  =\frac{P(D^\pm; \Omega)}{|D^\pm|}.
 \end{equation}
\end{lemma}
\begin{proof}
    Note that, 
    \begin{equation}\label{algeb ineq} 
    \min\left\{ \frac{a}{c}, \frac b d \right\} \le \frac{a+b}{c+d} \le \max \left\{ \frac{a}{c}, \frac b d \right\}\qquad \text{for $a, b, c, d > 0$} 
    \end{equation}
    and equality holds if and only if $ad=bc$.
   Moreover, for any $\omega_1, \omega_2$ satisfying \eqref{Nomega} one has $|\omega_1|, |\omega_2| \le |\Omega|/2$. 
   This gives the equality \eqref{charact 2}. 
   Furthermore, if $(D^+, D^-)$ is a minimizer for $\widetilde h_{\mathcal N}(\Omega)$ such that $D^+, D^- \ne \emptyset$, then $(D^+, \emptyset)$ and $(D^-, \emptyset)$ satisfy \eqref{Nomega} and equalities hold in \eqref{algeb ineq}, whence
   \[ \widetilde h_{\mathcal N}(\Omega)= \frac{P(D^+; \Omega)+ P(D^-; \Omega)}{|D^+ \cup D^-|} = \frac{P(D^+; \Omega)}{|D^+|} = \frac{P(D^-; \Omega)}{|D^-|}. \qedhere\]
\end{proof}
The proof of the next result is based on the coarea formula, in the spirit of \cite{KF,L97,BF25}.
\begin{theo}\label{thm cheeger}
    \begin{itemize}
    \item[(i)]$\widetilde h_{\mathcal N}(\Omega)$ and $h_{\mathcal N}(\Omega)$ are attained and $\lambda_{1, 1}= \widetilde h_{\mathcal N}(\Omega)=h_{\mathcal N}(\Omega).$
    \item[(ii)] If $u$ is a minimizer for $\lambda_{1,1}$ and
    \begin{align}\label{omegats}
    \omega_t^\pm:=\{ x \in \Omega: \, u^\pm (x) >t\} \qquad  \text{ for } t \in \R,         
    \end{align}
    then, for almost every $t \in (0, \norm{u}_\infty)$, 
    $(\omega_t^+, \omega_t^-)$ is a minimizer for $\widetilde h_{\mathcal N}(\Omega)$ and $\omega_t^\pm$ is a minimizer for $h_{\mathcal N}(\Omega)$ whenever $\omega_t^\pm\ne \emptyset$. 
    \item[(iii)] If $(D^+, D^-)$ is a minimizer for $\widetilde h_{\mathcal N}(\Omega)$ and $v:=\alpha \mathds{1}_{D^+}+\beta \mathds{1}_{D^-}$ for $\alpha, \beta \in \R$ such that $\alpha, \beta \ne 0$ and $\alpha\beta \le 0$, then $v$ is a minimizer for $\lambda_{1,1}$.
    \end{itemize}
\end{theo}
\begin{proof}
Let $u$ be a minimizer for $\lambda_{1,1}.$

\medskip

\underline{\textit{Step 1.}} Let us show that $\lambda_{1,1} \ge \widetilde h_{\mathcal N}(\Omega)$.  By the coarea formula \cite[Theorem~3.40]{AFP}, 
\begin{align}\label{absuco}
|Du|(\Omega)
=|D|u||(\Omega)
= \int_{-\infty}^{+\infty} P(\{|u|>t\}; \Omega) \, dt
= \int_0^{+\infty} P(\{|u|>t\}; \Omega) \, dt,
\end{align}
where we used that $P\big(\{|u|>t\};\Omega\big)=P(\Omega;\Omega)=0$ for $t<0.$ Then, using \eqref{omegats} and \eqref{absuco},
\begin{align} \label{absu2}
|Du|(\Omega) &= \int_{0}^{+\infty} P(\omega_t^+; \Omega)\,dt + \int_{0}^{+\infty} P(\omega_t^-; \Omega)\, dt=\int_0^{+\infty} \frac{P(\omega_t^+; \Omega) + P(\omega_t^-; \Omega)}{|\omega_t^+ \cup \omega_t^-|} |\omega_t^+ \cup \omega_t^-|\,dt. 
\end{align}
Note that $|\omega_t^+|= |\{ u >0 \}|- |\{0 < u \le t\}|$, $|\omega_t^-|=|\{ u<0\}| - |\{ -t \le  u <0\}|,$ and $|\Omega \setminus (\omega_t^+ \cup \omega_t^-)|= |\{ |u| \le t\}|.$ Hence, 
\begin{align*} ||\omega_t^+|-|\omega_t^-|| &= \big||\{ u>0\}|-|\{0 < u \le t \}| - |\{u<0\}| + |\{-t \le u <0\}|\big|\\
& \le |\{ u =0\}| + |\{0 < u \le t \}| +  |\{-t \le u <0\}|
= |\{ |u| \le t\}|= |\Omega \setminus (\omega_t^+ \cup \omega_t^-)|.
\end{align*}
This implies that $(\omega_t^+, \omega_t^-)$ satisfy \eqref{Nomega}. Then, by \eqref{absu2} and the definition of $\widetilde h_{\mathcal N}(\Omega)$,
\begin{align*}
|Du|(\Omega)&\ge \widetilde h_{\mathcal N}(\Omega) \int_{0}^{+\infty} |\omega_t^+ \cup \omega_t^-|\,dt= \widetilde h_{\mathcal N}(\Omega) \int_\Omega |u|\,d\mathcal{H}^N;
\end{align*}
namely, $\lambda_{1,1} \ge \widetilde h_{\mathcal N}(\Omega)$.

\medskip

\underline{\textit{Step 2.}} Let us show that $\lambda_{1,1 } \le \widetilde h_{\mathcal N}(\Omega)$: Let $D_n^+, D_n^-$ be a minimizing sequence for $\widetilde h_{\mathcal N}(\Omega)$, namely, $D_n^+, D_n^-$ are disjoint sets such that 
\begin{equation}\label{cond D} 
||D_n^+|-|D_n^-|| \le |\Omega \setminus (D_n^+ \cup D_n^-)| 
\quad \text{ and }\quad 
\frac{P(D_n^+, \Omega)+P(D_n^-, \Omega)}{|D_n^+ \cup D_n^-|} \to \widetilde h_{\mathcal N}(\Omega).
\end{equation}
Let $v_n:=\frac{\mathds{1}_{D_n^+} - \mathds{1}_{D_n^-}}{|D_n^+ \cup D_n^-|},$ then $|Dv_n|(\Omega)= \frac{P(D_n^-; \Omega) +P(D_n^+; \Omega)}{|D_n^+ \cup D_n^-|}.$ Thus $v_n$ is bounded in $BV(\Omega)$ and there exists $v \in BV(\Omega)$ such that, passing to a subsequence, $v_n \to v$ in $L^1(\Omega)$. Therefore,
\[
 |Dv|(\Omega) \le \lim_{n\to\infty}|D v_n|(\Omega) = \lim_{n\to\infty}\frac{P(D_n^+;\Omega) + P(D_n^-; \Omega)}{|D_n^+ \cup D_n^-|},
\qquad  
\int_\Omega |v|\,d\mathcal{H}^N = \lim_{n\to\infty}\int_\Omega |v_n|\,d\mathcal{H}^N=1,
\]
and $v_n \in \mathfrak{N}(\Omega)$, by \eqref{cond D}. Then, by \eqref{charact M} and Fatou's Lemma, 
\[
\int_\Omega \sgn_-(v)\,d\mathcal{H}^N \le \liminf_{n\to\infty}\int_\Omega \sgn_-(v_n)\,d\mathcal{H}^N \le 0 \le \limsup_n\int_\Omega \sgn_+(v_n)\,d\mathcal{H}^N \le \int_\Omega \sgn_+(v)\,d\mathcal{H}^N. 
\]
Therefore, $v \in \mathfrak{N}(\Omega)$ and $\lambda_{1,1 } \le |Dv|(\Omega) \le  \lim_{n\to\infty}|Dv_n|(\Omega)=\widetilde h_{\mathcal N}(\Omega),$ as claimed. 

\medskip

\underline{\textit{Step 3.}} Conclusion: By the previous steps and Lemma~\ref{lemma charac2}, $\lambda_{1,1}=\widetilde h_{\mathcal N}(\Omega)=h_{\mathcal N}(\Omega)$. The fact that $(\omega_t^+, \omega_t^-)$ is a minimizer for $\widetilde h_{\mathcal N}(\Omega)$ for almost every $t \in (0, \norm{u}_\infty)$ can be deduced from \eqref{absu2}; indeed, if this does not hold for all $t\in {\mathcal J}\subset (0,\|u\|_\infty)$ and ${\mathcal J}$ has positive measure, then 
\begin{align*}
\widetilde h_{\mathcal N}(\Omega)=\lambda_{1,1}=|Du|(\Omega) =\int_0^{+\infty} \frac{P(\omega_t^+; \Omega) + P(\omega_t^-; \Omega)}{|\omega_t^+ \cup \omega_t^-|} |\omega_t^+ \cup \omega_t^-|\,dt
>
\widetilde h_{\mathcal N}(\Omega)\int_\Omega|u|\,d\mathcal{H}^N=\widetilde h_{\mathcal N}(\Omega), 
\end{align*}
a contradiction. Hence $(ii)$ holds and $\widetilde h_{\mathcal N}(\Omega)$ is attained as a consequence of  Proposition~\ref{lambda 11 att}. Finally, to see $(iii)$, let $(D^+, D^-)$ be a minimizer of $\widetilde h_{\mathcal N}(\Omega)$ and let $v:=\alpha \mathds{1}_{D^+}+\beta \mathds{1}_{D^-}$ for some $\alpha, \beta \in \R$ such that $\alpha, \beta \ne 0$ and $\alpha\beta \le 0$. Since $\alpha\beta \le 0$ and $(D^+, D^-)$ satisfies \eqref{Nomega}, we have that $v \in \mathfrak {N}(\Omega)$.
We may assume that $D^+$ and $D^-$ are nonempty (otherwise the claim follows from the definition of $\widetilde h_{\mathcal N}(\Omega)$).  Then, by \eqref{algeb ineq},
\begin{align}
\min&\left\{\frac{P(D^+;\Omega)}{|D^+|},\frac{P(D^-;\Omega)}{|D^-|} \right\}
\leq 
\frac{P(D^+;\Omega)+P(D^-;\Omega)}{|D^+|+|D^-|}=
\widetilde h_{\mathcal N}(\Omega)\notag\\
&=\lambda_{1,1}
\leq \frac{|D v|(\Omega)}{\int_{\Omega}|v|\,d\mathcal{H}^N}=\frac{|\alpha|P(D^+;\Omega)+|\beta|P(D^-;\Omega)}{|\alpha||D^+|+|\beta||D^-|} 
\leq \max\left\{\frac{P(D^+;\Omega)}{|D^+|},\frac{P(D^-;\Omega)}{|D^-|} \right\}.\label{iiiclaim}
\end{align}
Then, by Lemma~\ref{lemma charac2}, $\frac{P(D^\pm;\Omega)}{|D^\pm|}=\widetilde h_{\mathcal N}(\Omega)=h_{\mathcal N}(\Omega)$. Hence, we have equalities in \eqref{iiiclaim} and $v$ is a minimizer for $\lambda_{1,1},$ as claimed in $(iii)$. 
\end{proof}

We are ready to show Theorem~\ref{theo charact lambda intro}.
\begin{proof}[Proof of Theorem~\ref{theo charact lambda intro}]
It follows from Lemma~\ref{lemma charac2} and Theorem~\ref{thm cheeger}.
\end{proof}

\subsection{Regularity of minimizers: a geometric measure theory perspective}\label{reg:app}

We first introduce some notation. By the Riesz representation theorem, if $u \in BV(\Omega)$ then there exists a Radon measure $D u=(D_1 u , \dots, D_N u)$ such that
\begin{align*}
    \int_\Omega u \frac{\partial \varphi}{\partial x_i} = -\int_\Omega \varphi \, dD_iu \quad \text{ for all } \varphi \in C_c^\infty(\Omega) \text{ and } i=1, \dots, N,
\end{align*}
see \cite[Proposition 3.6]{AFP}.  The vector-valued measure $Du$ may be written as $Du = \sigma |Du|$ for a $|Du|$-measurable function $\sigma:\Omega \rightarrow \mathbb{R}^N$  satisfying $|\sigma|=1$ $|Du|$-almost everywhere, see \cite[Theorem~5.1]{evansgariepy}.  Moreover, $D_i u = \sigma_i |Du|$.  In the case that $u = \chi_E \in BV(\Omega)$ for some $E \subset \mathbb{R}^N$, we say that $E$ is a set of finite perimeter in $\Omega$. We use the notation $D\chi_E = \sigma_E|D\chi_E|$, where now $\sigma_E$ may be thought of as the measure-theoretic outer normal to $E$.  The measure $|D\chi_E|$ is carried by the \textit{reduced boundary} of $E$, written $\partial^*E$, in the sense that $|D\chi_{E}|(\Omega \setminus \partial^*E) = 0$ (see the Remark after \cite[Definition 5.4]{evansgariepy}), where 
\begin{align*}
    \partial^*E:= \left\{x \in \Omega\::\: |D\chi_E|(B_r(x) \cap \Omega)>0\text{ for all }r>0, \lim_{r\rightarrow 0^+}\frac{D\chi_E(B_r(x))}{|D\chi_E|(B_r(x))}= \sigma_E(x), \text{ and } |\sigma_E(x)|=1\right\}.
\end{align*}
That is, the reduced boundary $\partial^*E$ represents, in some sense, Lebesgue points of $\sigma_E$ with respect to the measure $|D\chi_E|$.  Furthermore, for any open $U\subset \Omega$, $|D\chi_E|(U) = \mathcal{H}^{N-1}(U \cap \partial^*E)$, see \cite[Theorem~5.15]{evansgariepy}.

For any set $E \subset \mathbb{R}^n$, we define the measure-theoretic boundary $\partial^ME$, also called the \emph{essential boundary}, as follows.
\begin{align*}
    \partial^ME:= \left\{x \in \mathbb{R}^N: \lim_{r \rightarrow 0^+}\frac{|E \cap B_r(x)|}{|B_1(0)| r^N} \not= 0,1\right\}. 
\end{align*}
That is, $\partial^ME$ includes all the points where the density limit does not exist and all the points where the limit exists, but it is not in $\{0,1\}$. For sets of locally finite perimeter, $\mathcal{H}^{N-1}(\partial^ME \setminus \partial^*E) = 0$ by Federer's theorem (see \cite[Theorem 16.2]{Maggi12}).

Recall that, if $\Omega\subset \mathbb{R}^N$ is a Lipschitz domain, then, by Rademacher's Theorem~\cite[Theorem~3.2]{evansgariepy}, the boundary $\partial \Omega$ is differentiable $\mathcal{H}^{N-1}$-almost everywhere.  For $x \in \partial \Omega$ such that $\partial \Omega$ is differentiable, we use the notation $T_{x}\partial \Omega$ to denote the affine linear hyperplane tangent to $\partial \Omega$ at $x\in \partial\Omega \cap U$.

In the following theorem~we show that a minimizer of $h_{\mathcal N}(\Omega)$ has an analytic, constant-mean-curvature boundary outside a negligible singular set, which is absent for $N \leq 7$ and of Hausdorff dimension (denoted by $\dim_{\mathcal{H}}$) at most $N-8$ otherwise. 
     
\begin{theo}\label{thm reg} 
    Let $N\geq 2$ and let $A\subset \Omega$ be a minimizer for
    \begin{align*}
    h_{\mathcal N}(\Omega)&:= \inf \left\{ \frac{P(\omega; \Omega)}{|\omega|}:\, \omega\subset \Omega, \ 0 <|\omega| \le \frac{|\Omega|}{2} \right\}= \frac{P(A; \Omega)}{|A|}. 
    \end{align*}
    Then, the following statements hold.
    \begin{enumerate}
        \item The reduced boundary $\partial^*A$ is locally the graph of an analytic function $u:\mathbb{R}^{N-1}\rightarrow \mathbb{R}$ which satisfies the constant mean curvature equation. In particular, for every $x \in\partial^*A$ there exists a neighborhood $0<r_x$ and a function $u:\mathbb{R}^{N-1}\rightarrow \mathbb{R}$ such that, after translating $x \mapsto 0$ and rotating so that our outer unit normal satisfies $\sigma_A(0) = -\vec{e}_N$, then $\partial A \cap B_{r_x}^{N-1}(0) \times (-r_x, r_x) = \text{graph}_{B_{r_x}(0)}(u)$ and $u$ satisfies $u(0)=0$, $\nabla u(0)=0$, and 
            \begin{align}\label{mce1}
                \operatorname{div}\left(\frac{\nabla u}{\sqrt{1+|\nabla u|^2}}\right) &= h_{\mathcal N}(\Omega) \quad \text{in }B_{r_x}^{N-1}(0),\qquad \text{if $|A|<\frac{|\Omega|}{2}$},
\end{align}
and
\begin{align*}               
        \left|\operatorname{div}\left(\frac{\nabla u}{\sqrt{1+|\nabla u|^2}}\right)\right| &\le h_{\mathcal N}(\Omega) \quad \text{in }B_{r_x}^{N-1}(0),\qquad \text{if $|A|=\frac{|\Omega|}{2}$}.
            \end{align*}
        \item The singular set $\Sigma(A;\Omega):= \partial^MA \setminus \partial^*A$ satisfies that
            \begin{enumerate}
                \item if $N\in \{2,3,4,5,6,7\}$, then $\Sigma(A;\Omega)=\emptyset$.
                \item if $N \ge 8$, then $\dim_{\mathcal{H}}(\Sigma(A;\Omega)) \le N-8.$ When $N=8$ we have, additionally, that $\Sigma(A;\Omega)$ has no accumulation points in $\Omega$. 
         \end{enumerate}
        \end{enumerate}
\end{theo}

\begin{proof}
    If $A$ is a minimizer for \eqref{charact 2}, then certainly $A$ is a perimeter minimizer within the smaller class of sets $B\subset \Omega$ such that $|B|=|A|$.  In particular, it is a minimizer under volume-preserving deformations (see \cite[Section 17.3 and Section 17.5]{Maggi12} for definitions and more details).
    
    Therefore, by \cite[Example 21.3]{Maggi12}, there exist a $0\leq \Lambda<\infty$ and $r_0>0$ such that 
    \begin{align*}
        |D\chi_{A}|(B_r(x)) \le |D\chi_F|(B_r(x)) + \Lambda |B_r(x)|
    \end{align*}
    for every set of finite perimeter $F$ such that $F \Delta A \subset \subset B_r(x)$ and every $B_r(x)\subset \Omega$ with $0<r<r_0$. That is, $A$ is a $(\Lambda, r_0)$-perimeter minimizer (a kind of local almost-minimizer of perimeter, see \cite[(21.2)]{Maggi12}). By \cite[Theorem~21.8]{Maggi12}, $\partial^*A$ is locally the graph of a $C^{1,\gamma}$ function for some $0<\gamma<1$.  By \cite[Theorem~27.4]{Maggi12}, it is the graph of an analytic function.  Moreover, \cite[Theorem~17.20]{Maggi12} gives that $A$ is a set whose boundary has constant distributional mean curvature in the sense that there exists a $\lambda \in \mathbb{R}$ such that
    \begin{align}\label{e:cmc eq}
        \int_{\Omega} \operatorname{div}_{{A}}(\phi)\, d|D\chi_{{A}}| = \lambda \int_{\Omega} \phi \cdot \sigma_{A} \, d|D\chi_{{A}}|
        \qquad \text{ for any $\phi \in C^{\infty}_c(\Omega; \mathbb{R}^N)$,}
    \end{align}
        where $\operatorname{div}_{A}$ is the boundary divergence operator on ${A}$ given by 
    \begin{align*}
\operatorname{div}_{A}(\phi(x)):= \operatorname{div}(\phi(x)) - \sigma_{A}(x)\cdot D\phi(x)\sigma_A(x) \qquad \text{ for all }\phi \in C^{\infty}_c(\Omega;\mathbb R^N)\text{ at any $x\in \partial^*A$,}
    \end{align*}
    where $D \phi$ is the Jacobian of $\phi,$ see \cite[Theorem~17.5]{Maggi12} for further details. 
    
Let $\phi\in C^{\infty}_c(\Omega;\mathbb{R}^N)$ (extended trivially to $\mathbb R^N$) and let $\Phi_t:\mathbb R^N\to \mathbb R^N$ be a local variation associated with $\phi$ (see \cite[Section 17.3]{Maggi12} for definitions); in particular, the Taylor’s expansion $\Phi_t(x)=x+t\phi(x)+O(t^2)$ holds uniformly on $\mathbb R^N$ as $t\to 0$. Furthermore, by \cite[Theorem~17.5 and Proposition 17.8]{Maggi12},
\begin{align}\label{Meq}
\frac{d}{dt}P(\Phi_t(A);\Omega)\Big|_{t=0}=\int_{\mathbb R^N} \operatorname{div}_A(\phi)\, d|D\chi_{A}|\qquad \text{ and }\qquad 
\frac{d}{dt}|\Phi_t(A)|\Big|_{t=0}=\int_{\mathbb R^N} \phi \cdot \sigma_A\, d|D\chi_{A}|.
\end{align}
Now we consider two cases.

\medskip

\underline{Case I:} Assume that $|A|<|\Omega|/2$. Then, $|\Phi_t(A)|<|\Omega|/2$ for all sufficiently small $|t|$ (see \cite[(17.31)]{Maggi12}); hence,
\begin{align*}
0&=\frac{d}{dt}\left(\frac{P(\Phi_t(A);\Omega)}{|\Phi_t(A)|}\right)\Bigg|_{t=0} = \frac{\int_{\mathbb R^N} \operatorname{div}_A(\phi)\, d|D\chi_{A}|}{|A|} - \frac{P(A;\Omega)}{|A|^2}\int_{\mathbb R^N} \phi \cdot \sigma_A \, d|D\chi_{A}|\\
  & = \frac{\lambda\int_{\mathbb R^N} \phi\cdot \sigma_A\, d|D\chi_{A}|}{|A|} - \frac{P(A;\Omega)}{|A|^2}\int_{\mathbb R^N} \phi \cdot \sigma_A\, d|D\chi_{A}|.
\end{align*}
Then $\lambda = \frac{P(A;\Omega)}{|A|}$ and $\partial^*A$ is a constant mean curvature hypersurface with constant $\lambda = \frac{P(A;\Omega)}{|A|}$.

    \medskip
    
\underline{Case II:} Suppose that $|A|=|\Omega|/2$. Then, for any $\phi \in C^{\infty}_c(\Omega;\mathbb{R}^n)$ such that \begin{align}\label{e:volume shrinking condition}
    \int_{\mathbb R^N} \phi\cdot \sigma_Ad|D\chi_A|<0,
\end{align}
the sets $\Phi_t(A)$ are not admissible for $t<0$, only for $t>0$ (see \cite[(17.31)]{Maggi12}).  Thus, taking a one-sided variation, the minimality of $A$ gives, using \eqref{Meq}, the inequality
\begin{align*}
     0 & \le \liminf_{t \rightarrow 0^+}\frac{1}{t}\left(\frac{P(\Phi_t(A);\Omega)}{|\Phi_t(A)|} - \frac{P(A;\Omega)}{|A|}\right)= \frac{\lambda\int_{\mathbb R^N} \phi\cdot \sigma_Ad|D\chi_{A}|}{|A|} - \frac{P(A;\Omega)}{|A|^2}\int_{\mathbb R^N} \phi \cdot \sigma_A d|D\chi_{A}|.
\end{align*}
Recalling \eqref{e:volume shrinking condition} and rearranging, we see $\lambda \le \frac{P(A;\Omega)}{|A|}$.  To obtain the complementary inequality, we simply note that, if $\lambda_A$ is the distributional mean curvature relative to $A$, then $\lambda_A =- \lambda_{A^c \cap \Omega}$, since  $\sigma_A = -\sigma_{A_c \cap \Omega}$. Therefore, since $A^c \cap \Omega$ is also a minimizer of $h_{\mathcal N}(\Omega)$, the previous argument gives $\lambda_{A^c \cap \Omega} \le \frac{P(A^c;\Omega)}{|A^c \cap \Omega|} = \frac{P(A;\Omega)}{|A|}$. 

\medskip

The only remaining part of claim $1$ is to show \eqref{mce1}. Let $x\in \partial^*A$ be given.  By translating $x \mapsto 0$ and rotating so that $\sigma_A(0) = -\vec{e}_N$, we write $\partial^*A \cap B_{r_x}^{N-1}(0) \times (-r_x, r_x)$ as a graph of a function $u:B_{r_x}^{N-1}(0)\rightarrow \mathbb{R}$ with $u(0)=0$ and $|\nabla u(0)|=0$ for some $r_x>0$.  Thus, for any $\varphi \in C^\infty_c(B_{r_x}^{N-1}(0))$, we consider local variations generated by $\phi(x',x_N):= \vec{e}_N\varphi(x')$.  In this case, \eqref{e:cmc eq} becomes 
\begin{align}\label{flcv}
    -\int_{B_{r_x}^{N-1}(0)}\operatorname{div}\left(\frac{\nabla u}{\sqrt{1+|\nabla u|^2}}\right)\varphi \,d\mathcal{H}^{N-1} = \int\operatorname{div}_A(\phi)\, d|D\chi_A|
    = \lambda \int \phi \cdot \sigma_A \,d|D\chi_A| =  \int_{B_{r_x}^{N-1}(0)} -\lambda \varphi \,d\mathcal{H}^{N-1},\qquad 
\end{align}
where we have used that in the graphical case
\begin{align*}
    P(A; B_{r_x}^{N-1}(0) \times (-r_x, r_x)) = \int_{B_{r_x}^{N-1}(0)}\sqrt{1 + |\nabla u|^2}\, d\mathcal{H}^{N-1}
\end{align*}
and \eqref{Meq} to obtain the left hand side (see \cite[Proposition 11.16]{GiaquintaMartinazzi2012_BOOK}  for details in the general case) and 
\begin{align*}
    \qquad \sigma_A = \frac{1}{\sqrt{1 + |\nabla u|^2}}(-\nabla u, 1) \in \mathbb{R}^{N-1} \times \mathbb{R},
\end{align*} and the Area formula (see \cite[Theorem 3.8]{evansgariepy}) to obtain the right hand side. Hence, by \eqref{flcv}, we deduce that $-\operatorname{div}\left(\frac{\nabla u(x')}{\sqrt{1+|\nabla u(x')|^2}}\right)\Big|_{x'=0}=-\lambda$ and claim $1$ follows.

Now that we know that, for any minimizer $A$ of \eqref{charact 2}, $A$ is a volume-constrained perimeter minimizer, the statements in 2. follow from Simon's Theorem~on the non-existence of singular minimizing cones which are smooth away from the origin in $\mathbb{R}^N$ for $N=2,3,4,5,6,7$ and from Federer's dimension reduction principle (see \cite[Theorem~28.1]{Maggi12} for full details).
\end{proof}

We are ready to show Theorem~\ref{regn2}, stated in the introduction, see also \cite{C89} for a different proof.

\begin{proof}[Proof of Theorem~\ref{regn2}]
Let $N=2$ and let $D$ be a minimizer for $h_{\mathcal N}(\Omega)$ such that 
$|D|<|\Omega|/2$; in particular, $0<h_{\mathcal N}(\Omega)< \infty$. The fact that any portion of the boundary $\partial D$ (which coincides with $\partial^* D$ since there is no singular set for $N=2$) is part of a circle with radius $\frac{1}{\lambda_{1,1}}$ follows from Theorem~\ref{thm reg}.

On the other hand, the orthogonality of the free boundary for volume-constrained perimeter minimizers in $C^1$ domains is well-known, but we could not locate a clean reference. Hence, for completeness, we provide a sketch of the proof in the $2$-dimensional case.  

  Let $\Gamma: [0,1] \rightarrow \overline{\Omega}$ be a continuous curve such that $\Gamma((0,1)) \subset \partial D \cap \Omega$ and $\Gamma(0), \Gamma(1) \in \partial \Omega$.  Let $T_{\Gamma(0)}\partial\Omega$ denote the tangent line to $\partial\Omega$ at $\Gamma(0)$ and let $C := \{(x, y): x^2 + y^2 = r^2\}$ for $0<r = h_{\mathcal N}(\Omega)^{-1}<\infty$. 
Without loss of generality (using rotations, reflections, and translations), we may assume that
\begin{enumerate}
\item $T_{\Gamma(0)}\partial \Omega=L := \{(x, t): x\in \mathbb{R}\}$ for some height $t\in[-r,r]$,
\item $\Gamma(0)= r(\cos(\theta_0), \sin(\theta_0))$ for some $\theta_0 \in [-\pi,\pi]$,
\item there is $\theta_1>0$ such that $\Gamma([0,1])=\{r(\cos(\theta), \sin(\theta))\::\:\theta \in (\theta_0,\theta_1)\}$.  
\end{enumerate}

Observe that, under this situation, the minimizer $D$ always lies inside the circle $C$ (otherwise one could construct a better competitor by changing a portion of the boundary $\partial D$ with a chord inside the circle, this shortens the perimeter and increases the area).

Now we show that, if there is a neighborhood $U$ around $\Gamma(0)$ such that $\partial \Omega \cap U$ is $C^2$, then $\Gamma$ meets $\partial \Omega$ orthogonally, or equivalently, that $\theta_0=0$ or $\theta_0=\pi$. For this, let us construct a competitor by modifying a portion of the boundary $\partial D$ with a straight vertical line. For $\theta>0$ small, let
\begin{align*}
    D_{\theta} := \begin{cases}
    D \setminus \left\{(x, y) \in \Omega\cap B_r(0)\::\: x\ge r\cos(\theta_0+\theta)\right\},&\text{ if } \theta_0\geq 0,\\
    D \cup \Big\{(x, y) \in \Omega\cap B_r(0)^c\::\: r\cos(\theta_0)\le x\le r\cos(\theta_0+\theta), &\text{ if } \theta_0<0.
    \end{cases}
\end{align*}
Namely, if $\theta_0\geq 0$, then we \emph{remove} a portion of $D$, whereas, if $\theta_0<0$, we \emph{add} a portion to $D$, see Figure \ref{Dthetafig}.  Note that, since $|D|<|\Omega|/2$, then $|D_\theta|<|\Omega|/2$ for $\theta$ sufficiently small, and therefore $D_\theta$ is a competitor for $h_{\mathcal N}(\Omega)$. 

\begin{figure}[h!]
    \centering
    \begin{subfigure}{0.4\linewidth}
        \centering
        \includegraphics[width=\linewidth]{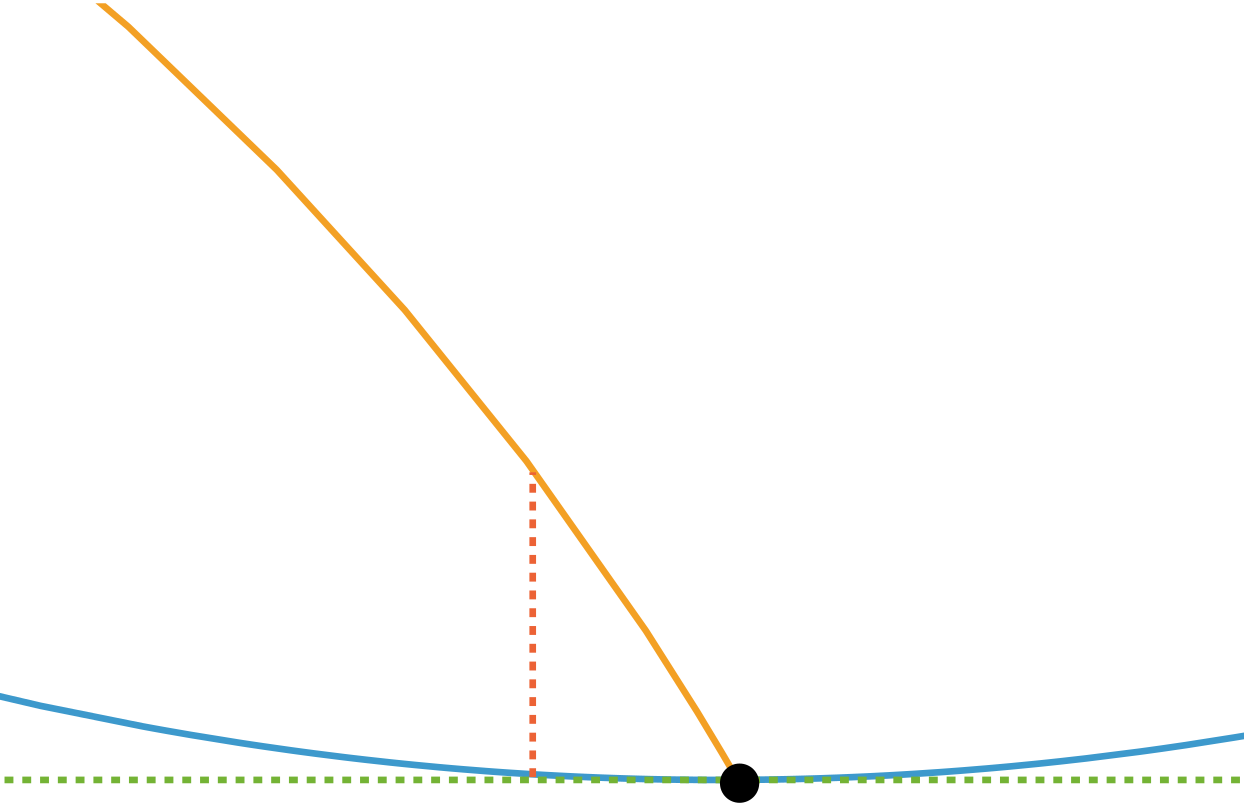}
        \caption{$\theta_0\geq 0$}
    \end{subfigure}
    \hspace{0.5cm}
    \begin{subfigure}{0.4\linewidth}
        \centering
        \includegraphics[width=\linewidth]{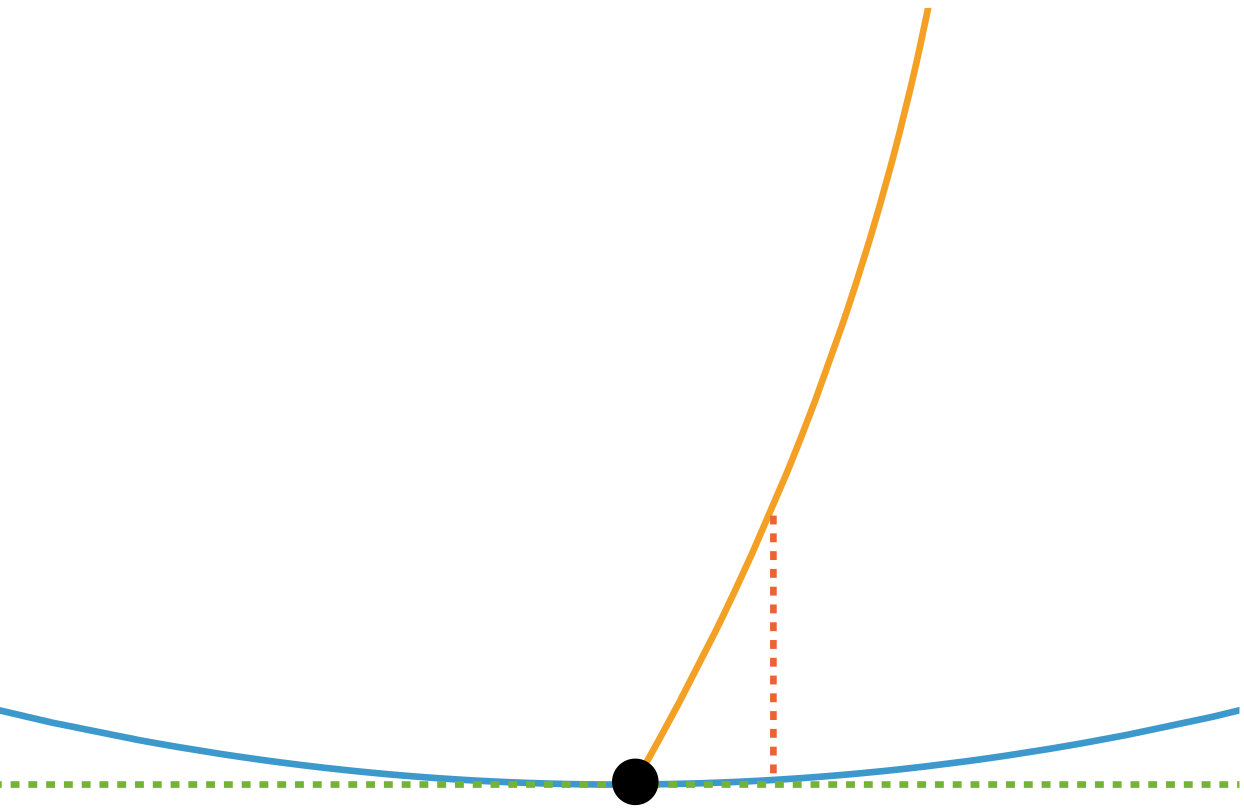}
        \caption{$\theta_0<0$}
    \end{subfigure}
    \caption{The construction of the set $D_\theta$.  The black dot is $\Gamma(0)=r(\cos(\theta_0),\sin(\theta_0))$, the horizontal dotted line is $L$, the horizontal curved line (blue) is $\partial \Omega$, the vertical curved line (orange) is $\Gamma$, and the vertical dotted line marks the boundary of the portion of the domain that is being added (if $\theta_0<0$) or removed (if $\theta_0\geq 0$) to define $D_\theta$.}
    \label{Dthetafig}
\end{figure}

Observe that, since the boundary $\partial \Omega$ is $C^2$ around $\Gamma(0)$, it is almost flat close to $\Gamma(0)$. Hence, using a Taylor expansion,
\begin{align*}
     P(D;\Omega) - P(D_{\theta};\Omega) & \ge r\theta- r|\sin(\theta_0+\theta) - \sin(\theta_0)|
     +o(\theta)\ge \frac{1}{2}r(1 -|\cos(\theta_0)|)\theta\qquad \text{as $\theta\to 0$. }
 \end{align*}
Similarly, we estimate the change in area 
\begin{align*}
    \Big||D| - |D_{\theta}|\Big| & = \frac{r^2}{2}
   \left|\sin(\theta_0 +\theta) - \sin(\theta_0)\right|\left|\cos(\theta_0 +\theta) - \cos(\theta_0)\right| + o(\theta) = o(\theta)\qquad \text{as $\theta\to 0$. }
\end{align*}
Here we are using that the difference in area has, asymptotically, a triangular shape, see Figure \ref{Dthetafig}.

Hence,
\begin{align*}
    \frac{P(D_{\theta};\Omega)}{|D_{\theta}|}
    \le \frac{P(D;\Omega) - \frac{r}{2} (1 -|\cos(\theta_0)|)\theta}{|D|+o(\theta)} 
    = \frac{P(D;\Omega)}{|D|} - \frac{r(1 -|\cos(\theta_0)|)\theta }{2(|D|+o(\theta))} + o(\theta)\qquad \text{as $\theta\to 0$. }
\end{align*}
Since $D$ is a minimizer, this implies that $1 -|\cos(\theta_0)|=0$, namely $\theta_0$ must be equal to $0$ or $\pi$, which implies that $\partial D$ meets $\partial \Omega$ orthogonally.
\end{proof}

\section{Explicit constants and minimizers}\label{sec:minimizers}

In this section we explore the explicit value of $h_{\mathcal N}(\Omega)$ and the shape of minimizers in some particular geometries. 

\subsection{Minimizers in convex domains}\label{sec:convex}

As above, let $\Omega \subset \mathbb{R}^N$ be a bounded, Lipschitz domain.  For any $m \in [0, |\Omega|)$, the \emph{isoperimetric profile} of $\Omega$ at volume $m$ is given by  
\begin{align}\label{isoprof}
\mathcal I(m):=\inf\{ P(E; \Omega): \, E \subset \Omega, \, |E|=m\}. 
\end{align}
In bounded Lipschitz domains, $\mathcal I(m)$ is attained for any $m \in [0, |\Omega|/2]$, see \cite[Theorem~9.3]{Ritore}. 
Let us recall the following result.
\begin{theo}[Theorem~2.8 in \cite{sz}]\label{concave}
Let $\Omega \subset \R^N$ be an open, bounded, strictly convex set with $C^{2, \alpha}$ boundary. Then the function $\mathcal{I}(m)$
is a concave function of $m$ on the interval $0 \le m \le |\Omega|$ and it is symmetric about $|\Omega|/2$. 
\end{theo}
\begin{remark}\label{strict}
In \cite[Theorem~1.1]{kuwert}  it is shown that $(\mathcal I(m))^{\frac N{N-1}}$ is concave. Thus, in particular, $\mathcal I$ is \emph{strictly} concave.
\end{remark}

The next result shows that $h_{\mathcal N}(\Omega)$ can be obtained in terms of $\mathcal I(|\Omega|/2)$.

\begin{theo}\label{case convex}
   Let $\Omega \subset \R^N$ be an open, bounded, strictly convex set with $C^{2, \alpha}$ boundary.  Then, 
   \begin{equation}\label{charact convex} 
   h_{\mathcal N}(\Omega)=\frac{2}{|\Omega|} \mathcal I(|\Omega|/2). 
   \end{equation}
   Moreover, $E$ is a minimizer for $\mathcal I(|\Omega|/2)$ if and only if $E$ is a minimizer for $h_{\mathcal N}(\Omega)$. 
\end{theo}
\begin{proof}
By Theorem~\ref{concave}, $\mathcal I$ is concave. Hence, for $0 < a < b\le |\Omega|/2$ and $t \in (0, 1)$ such that $a=tb$, we have 
\[ 
\mathcal I(a) = \mathcal I(tb) \ge t\mathcal I(b) + (1-t)\mathcal I(0) = t \mathcal I(b)=\frac{a}{b}\mathcal I(b). 
\]
Hence, $\frac{\mathcal I(a)}{a}\ge \frac{\mathcal I(b)}{b},$ i.e., the function $s \mapsto \mathcal I(s)/s$ is non-increasing in $(0, |\Omega|/2]$. 
Thus,
\[ \inf_{0 < s \le |\Omega|/2} \frac{\mathcal I(s)}{s} = \frac{\mathcal I(|\Omega|/2)}{|\Omega|/2} = \frac{2}{|\Omega|} \mathcal I(|\Omega|/2). \]

By \cite[Theorem~9.3]{Ritore}, $\mathcal I (s)$ is attained for any $s$. Thus, there exists a set $E$ such that $\mathcal I(|\Omega|/2)=P(E; \Omega)$ and $|E|=|\Omega|/2$. Therefore,
\[ 
\frac{2}{|\Omega|} \mathcal I(|\Omega|/2)=\frac{P(E; \Omega)}{|E|} \ge h_{\mathcal N}(\Omega) = \inf\left\{ \frac{P(D; \Omega)}{|D|}\::\:D\subset \Omega,\ 0 <|D| \le |\Omega|/2 \right\} \ge \inf_{0 < s \le |\Omega|/2} \frac{\mathcal I(s)}{s} = \frac{2}{|\Omega|} \mathcal I(|\Omega|/2),
\]
and the claim follows, also recalling Remark \ref{strict}.
\end{proof}

Recall the complete characterization of minimizers of $\mathcal I(|\Omega|/2)$ in \cite{A, BS}, see also \cite{Ros},
we have the following symmetry breaking result. It is a direct consequence of Theorems \ref{theo conv intro} and \ref{case convex}.
\begin{coro}\label{coro sphere}
    Let $\Omega=\{ x \in \R^N: \, |x| < L\}$ and $E^+:=\Omega \cap \{ x_1 >0\}$. 
    Then, the only  minimizer (up to rotations) of $h_{\mathcal N}(\Omega)$ is $E^+$. 
      Moreover, 
    \[ h_{\mathcal N}(\Omega)=\frac{2 \omega_{N-1}}{L \omega_N}
    =\frac{2\Gamma \left(\frac{N}{2}+1\right)}{\sqrt{\pi } \Gamma
   \left(\frac{N+1}{2}\right)} L^{-1},\]
 where $\omega_k=\frac{\pi^{\frac{k}{2}}}{\Gamma(\frac{k}{2}+1)}$ denotes the volume of the $k$-dimensional unitary ball. As a consequence, if $u_{p,q}$ is a minimizer of $\lambda_{p,q}(\Omega)$, then $u_{p,q}$ is not radially symmetric for all $(p,q)$ sufficiently close to $(1,1)$.
\end{coro}
The fact that the Neumann-Cheeger constant of balls is achieved at half balls was known, see \cite[Theorem 1]{C892}.

Next, let us now consider the unit cube $C_N=[0, 1]^N$ in dimension $N.$ 
\begin{theo}\label{cube}
    Let $\Omega=C_N$ and $E^+:=C_N \cap \{ x_1 >1/2\}$.
    Then, $h_{\mathcal N}(\Omega)=2$ 
   and $E^+$ is the only minimizer (up to rotations) of $h_{\mathcal N}(\Omega)$.
\end{theo}
\begin{proof}
Let $E^+$ as in the statement and recall that $\mathcal I_N(s):=\inf\{ P(E, C_N): \, E \subset C_N, \, |E|=s\}.$ By \cite{BB}, $\mathcal I_N (t) \ge 4t(1-t)$ for any $t \le 1/2$ and $N \ge 2.$ Then
\begin{align*}
\frac{P(E^+; C_N)}{|E^+|} \ge h_{\mathcal N}(C_N) \ge \inf_{0 \le s \le 1/2} \frac{\mathcal I_N(s)}{s} \ge 4 \inf_{0 \le s \le 1/2} (1-s) =2=\frac{P(E^+; C_N)}{|E^+|}.
\end{align*}
This implies that $h_{\mathcal N}(C_N)=2$ and $E^+$ is a minimizer for $h_N(C_N)$. 
By \cite{Hadwiger} (see also \cite[Corollary 2]{Ros}) we have that the sets achieving $\mathcal I_N(1/2)$ are unique. This concludes the proof.
\end{proof}
\begin{remark}
Theorem~\ref{cube} can be extended to more general rectangles $[a, b] \times [c,d ]$, see \cite{BB}. 
\end{remark}

Recalling Theorem~\ref{thm cheeger}, in terms of minimizers of $\lambda_{1,1}$ we have the following 
\begin{coro}\label{coro convex}
    Let $\Omega \subset \R^N$ be an open, bounded, strictly convex set with $C^{2, \alpha}$ boundary or let $\Omega=C_N:=[0,1]^N$. Let $u$ be a minimizer for $\lambda_{1,1}(\Omega)$. Then 
    $u= \alpha \chi_{\omega_1} - \beta \chi_{\omega_2}$ for some minimizers $\omega_1$ and $\omega_2$ of $\mathcal I(|\Omega|/2)$ and for some $\alpha, \beta \ge 0$ with $(\alpha, \beta) \ne (0,0)$. 
In particular, for any sign changing minimizer $u$  of $\lambda_{1,1}(\Omega)$, one has that $\int_\Omega \sgn(u)\,d\mathcal{H}^N=0$ and $u$ solves 
\begin{align}\label{probcoro}
-\operatorname{div} \left( \frac{Du}{|Du|} \right)= \lambda_{1,1}(\Omega) \sgn(u)\quad \text{ in } \Omega, \qquad \partial_\nu u =0 \quad \text{ on }\partial \Omega.   
\end{align} 
    \end{coro}
\begin{proof}
Let $u$ be a minimizer for $\lambda_{1,1}(\Omega)$.     By Theorem~\ref{thm cheeger}, we know that $(\{ u^+ > t\}, \{ u^- > t\})$ is a minimizer for $\widetilde h_{\mathcal N}(\Omega)$ for a.e. $t \in (0, \norm{u}_\infty)$. By Theorems ~\ref{case convex} (in the case of bounded, strictly convex, $C^{2,\alpha}$ convex sets) and~\ref{cube} (in the case of the cube), we have that 
\begin{align}\label{emptyor}
\text{$\{ u^+ > t\}$ and $\{ u^- > t\}$ are either empty or are minimizers for $\mathcal I(|\Omega|/2)$ for a.e. $t \in (0, \norm{u}_\infty)$.}
\end{align}
Without loss of generality, let us assume that $\{ u^+ > t\}$ is nonempty for a.e. $t \in (0, \norm{u}_\infty)$. Then, by \eqref{emptyor}, $|\{ u^+ > t\}|= |\Omega|/2$ for a.e. $t \in (0, \norm{u}_\infty)$. Since $\{ u^+>t \} \subset \{ u^+> s\}$ for any $t \ge s$, this shows that necessarily 
    \[
    \{ u^+> t\} = \{ u^+> s\} \quad \text{for a.e. $s, t \in (0, \norm{u}_\infty)$, }
    \]
    which yields that $u^+$ is a step function. Similarly, $u^-$ is also a step function or it is identically zero.   Now assume that $u$ changes sign.  Then, $u= \alpha \chi_{\omega_1} - \beta \chi_{\omega_2}$ for some minimizers $\omega_1$ and $\omega_2$ of $\mathcal I(|\Omega|/2)$ and for some $\alpha>0$ and $\beta>0$. By Theorem~\ref{thm cheeger}, $u$ is a minimizer of $\lambda_{1,1}(\Omega)$ and, by \eqref{emptyor}, $\int_\Omega \sgn(u)\,d\mathcal{H}^N=|\{u >0\}|- |\{u <0\}|= |\Omega|/2-|\Omega|/2=0$. Finally, by Theorem~\ref{main theo 2}, $u$ is a solution of \eqref{probcoro}.
\end{proof}
 \begin{remark}\label{rmk:notattained}
Corollary~\ref{coro convex} shows that, in general, $\lambda_{1,1}$ is not attained in $W^{1, 1}(\Omega)$.
 \end{remark}
\begin{remark}\label{rmk:not sign}
Let $\Omega \subset \R^N$ be an open, bounded, strictly convex set with $C^{2, \alpha}$ boundary, $\alpha >0$, and let $\omega$ be a minimizer for $\mathcal I(|\Omega|/2)$; in particular, $|\omega|=|\Omega|/2$.  By Theorems~\ref{thm cheeger} and~\ref{case convex}, we have that  $u = \alpha \chi_\omega$ is a minimizer for $\lambda_{1,1}(\Omega)$. However, $ \int_\Omega \sgn(u)=|\omega|>0.$ Moreover, by Theorem~\ref{main theo 2}, $u$ is a solution of \eqref{selection:prob} with $s:=\chi_{\omega}-\chi_{\Omega\backslash \omega}\neq \sign(u).$
\end{remark}

We are ready to show Theorem~\ref{thm convex intro}.
\begin{proof}[Proof of Theorem~\ref{thm convex intro}]
    The claim follows from Theorems~\ref{case convex} and~\ref{cube}. 
\end{proof}

\subsection{Minimizers in annuli}\label{sec:rad}

Now we turn our attention to annular domains.  Due to Corollary~\ref{coro sphere}, we have that minimizers of $h_{\mathcal N}(\Omega)$ in the ball are not radially symmetric. Next, we give a complete characterization of minimizers in the radially symmetric setting and use this to show that symmetry breaking also occurs in $N$-dimensional annuli. 

Let
\begin{align*}
\Omega := B_L(0) \setminus B_l(0)\qquad \text{ with $0 \le l < L$}
\end{align*}
and
    \begin{align*}
    \lambda_{rad}:= \inf \left\{|Du|(\Omega): \, \norm{u}_1=1, \, u \in \mathfrak{N}(\Omega), \, u \text{ is radially symmetric} \right\}. 
    \end{align*}
Arguing as in Theorem~\ref{thm cheeger} (using functional spaces of radially symmetric functions), we have that $\lambda_{rad}$ is attained and 
\begin{align*} 
\nonumber\lambda_{rad}
=h_{rad}(\Omega):= \inf\Big\{\frac{P(D; \Omega)}{|D|}: \, D\subset \Omega \text{ is radially symmetric and }  0<|D|\leq \frac{|\Omega|}{2}\Big\}.
\end{align*}

Let
\begin{align*}
A_{s,t}:=B_{t}(0) \setminus B_{s}(0) \subset \R^N.     
\end{align*}

\begin{lemma}\label{rad annulus}
Let $\Omega= A_{l,L}$ with $0 \le l < L$ in $\R^N$ and let $r^*:= \left(\frac{L^N+l^N}{2}\right)^{1/N}$. 
If $D$ is a minimizer of $h_{rad}(\Omega)$, then $D=A_{ l, r^*}$ or $D=A_{r^*, L}$. Moreover,
\begin{align}\label{hrad}
h_{rad}(\Omega)= 2^{1/N} N \frac{(L^N + l^N)^{\frac{N-1}{N}}}{L^N-l^N}.
\end{align}
\end{lemma}
\begin{proof}
    Let $F\subset \Omega$ be a radially symmetric set.  That is, in spherical coordinates $F = E \times \mathbb{S}^{N-1}$ for some $E \subset (l,L)$. Now, assume that $F$ is a set of finite perimeter in $\Omega$.  Since $\Omega = A_{l, L}$ is diffeomorphic to $(l,L) \times \mathbb{S}^{N-1}$, we may use \cite[Proposition 17.1]{Maggi12}  locally and the rotational symmetry of $F$ to see that $\partial^* F = \{x \in \mathbb{R}\,:\, |x| \in \partial^*E \subset (l,L) \}$ and
    \begin{align}\label{e:spherical set perimeter formula}
        P(F;\Omega) = \int_{\partial^*E \cap (l, L)} N\omega_Nx^{N-1}d\mathcal{H}^{0}(x)<\infty.
    \end{align}
In particular, $E$ is a set of finite perimeter in $(l, L)$ if $l>0$, and $E$ is a set of finite  perimeter in $(\ell, L)$ and  for any $\ell>0$ in the case $l=0$.  By the characterization of sets of finite perimeter in $\mathbb{R}$ (see \cite[Proposition 12.13]{Maggi12}), $E \cap A_{\ell,L}$ is $\mathcal{H}^1$-equivalent to a finite union of disjoint intervals $\cup_{i=1}^{N(\ell)}(a_i, b_{i})$ where $a_i \not = b_j$ for any $1\le i,j\le N(\ell)$.  In particular, for any $\ell>l$, $F \cap A_{\ell, L}= A_{\ell, L} \cap \bigcup_{i=1}^{N(\ell)}A_{a_i, b_i}$. Without loss of generality, we may assume that $a_{i+1}<b_{i+1}<a_i<b_i$ for all $i =1, ..., N(\ell)-1$.

Now, without loss of generality, we may assume that $l<a_1<L$ and either $b_{1}\ge L$ or $b_{1}<L$.  If $b_{1}<L$, then if we define $s^* \in (l, L)$ such that $|A_{l,s^*}|= |F|$, we claim that $P(A_{l, s^{*}};\Omega)\le P(F;\Omega)$ with equality if and only if $A_{l, s^*} = F$.  Clearly, if $A_{l, s^*} = F$, then equality holds for the perimeters.  If $F \not = A_{l, s^*}$, then $b_{1}>s^*$ since otherwise $F \cap A_{l, s^*} = F$ and the condition $|A_{l,s^*}|=|F|$ forces $A_{l,s^*}= F$. But, if $b_{1}>s^*$ it must be that $P(A_{l, s^{*}};\Omega)< P(F;\Omega)$ by \eqref{e:spherical set perimeter formula}. 

In the case that $b_1 \ge L$ and $a_1<L$, 
we have that $\text{P}(A_{l, s^*};\Omega) =\text{P}(A_{s^*, L};\Omega)<\text{P}(F;\Omega)$ unless $F = A_{s^*, L}$.  Thus, if $D$ is a minimizer for $h_{rad}(\Omega)$, then $D = A_{l,s_1}$ or $A_{s_2,L}$ where $l<s_1<s_2<L$ are chosen such that $|A_{l,s_1}| = |D| = |A_{s_2,L}|$.  The claim is now trivially verified by \eqref{e:spherical set perimeter formula}.  Namely, if $|D|<\frac{1}{2}|A_{l,L}|$, then $s_1 <s_2$ and $A_{l,s_1}$ is the volume-constrained perimeter-minimizer in the class of radial sets. On the other hand, if $|D|=\frac{1}{2}|A_{l,L}|$, then 
\[ s_1=s_2= r^*=\left(\frac{L^N + l^N}{2}\right)^{1/N} \]
 and the only volume-constrained perimeter minimizers within the class of radially symmetric sets are $A_{l, r^*}$ and $A_{r^*,L}$.  But, it is clear that
\begin{align*}
    \frac{d}{ds}\left(\frac{P(A_{l,s};\Omega)}{|A_{l,s}|}\right) = \frac{N\omega_N}{\omega_N} \frac{d}{ds}\left(\frac{s^{N-1}}{s^N - l^N}\right) = N\frac{s^{N-2}}{s^N-l^N}\left((N-1)-N\frac{s^N}{s^N-l^N}\right)<0 \quad \forall l<s<L.
\end{align*}
Therefore, $A_{l, r^*}$ and $A_{r^*,L}$ are the unique minimizers of $h_{rad}(\Omega)$.
\end{proof}

\begin{prop}\label{symm breaking}
Let $\Omega=B_L(0) \setminus B_l(0)$ with $0 \le l < L$ in $\R^N$. Then $h_{\mathcal N}(\Omega)<h_{rad}(\Omega).$ In particular, the minimizers of $h_{\mathcal N}(\Omega)$ are not radially symmetric. 
\end{prop}
\begin{proof}
Let $E^+:=\Omega \cap \{ x_1 >0\}$ and $E^-:=\Omega \cap \{ x_1 <0\}.$ 
Notice that
\[ L^{N-1} - l^{N-1} \le  L^{N-1} \le (L^N+ l^N)^{\frac{N-1}{N}}. \]
Moreover, by using Wendel's inequality \cite{Wendel}, 
\[ \frac{2 \omega_{N-1}}{\omega_N} = \frac{2 \Gamma(N/2+1)}{\sqrt{\pi}\Gamma((N+1)/2)}  \le  \sqrt{\frac{2(N+1)}{\pi}}  < 2^{1/N} N.   \]
By \eqref{hrad},
\[
h_{\mathcal N}(\Omega)\leq \frac{P(E^\pm; \Omega)}{(|\Omega|/2)} = \frac{2 \omega_{N-1}(L^{N-1} - l^{N-1})}{\omega_N(L^N- l^N)} < 2^{1/N} N \frac{(L^N + l^N)^{\frac{N-1}{N}}}{L^N-l^N}=h_{rad}(\Omega),
\]
as claimed.
\end{proof}

The following is a direct corollary of the previous results for the Neumann eigenfunctions of the $1$-Laplacian in an annulus.
\begin{coro}
Let $\Omega=A_{l, L}$, $0 \le l < L$ in $\R^N$. 
    Any minimizer $u$ of $\lambda_{rad}$ is such that, for almost every $t>0$, the sets $\omega_t^\pm:=\{ x \in \Omega: \, u^\pm (x) >t\}$ for $t \in \R,$ are of the form given in Lemma~\ref{rad annulus}. Moreover, minimizers of $\lambda_{1,1 }$ are not radially symmetric. 
\end{coro}

\subsection{Minimizers in some asymmetrical domains}\label{counterex}

Let 
\begin{align}\label{Lell}
\eps=\frac{1}{44}\quad \text{ and }\quad  \ell\in[1,2].
\end{align}

For $x_0=(x_1', x_2')\in \mathbb{R}^2$, $a, b \in \mathbb{R}_+$, we define the rectangle
\begin{align*}
    Q(a, b, x_0):= \left\{(x_1, x_2) \in \mathbb{R}^2: |x_1-x_1'|< \frac{a}{2},\ |x_2-x_2'|< \frac{b}{2}\right\}.
\end{align*}
Namely, $Q(a, b, x)$ is a rectangle of sides with lengths $a$ and $b$ centered at $x$.

Let
\begin{align}\label{Omegadef}
\Omega := \underbrace{Q(1,1, (-\tfrac{7}{2},0))}_{\text{the square on the left}} 
\quad\cup \underbrace{Q(4,4,(0,0))}_{\text{the square on the center}}
\cup \quad
\underbrace{Q(\ell,\ell, (3+\tfrac{\ell}{2},0))}_{\text{the square on the right}}
\quad\cup \quad
\underbrace{Q(6, \varepsilon, (0,0))}_{\text{the two bridges}}. 
\end{align}

For simplicity, let
\begin{align*}
S_L:=Q(1,1, (-\tfrac{7}{2},0)),\quad
S_C:=Q(4,4,(0,0)),\quad 
S_R:=Q(\ell,\ell, (3+\tfrac{\ell}{2},0)).
\end{align*}

Recall that $B_r(x)$ is the open ball centered at $x$ of radius $r$ and that $\lambda_{1,1}(\Omega)=h_{\mathcal N}(\Omega)$ (see Theorem~\ref{thm cheeger}). Let 
\begin{align*}
r:=(\lambda_{1,1}(\Omega))^{-1},\quad 
\rho:=\sqrt{r^2-(\frac{1}{88})^2}+2,  \quad 
E_1:=B_{r}((-\rho,0)) \cap \Omega,
\quad 
\text{ and }
\quad 
E_2:=B_{r}((\rho,0)) \cap \Omega.
\end{align*}
The circles $\partial B_{r}((-\rho,0))$ and $\partial B_{r}((\rho,0))$ are constructed in such a way that they have curvature $\lambda_{1,1}$ and they pass through intersection points of the bridge with the square in the center $S_C$, see Figure~\ref{hwe}, where the grey area is $E_2$.

Recall \eqref{isoprof}. By \cite{BB}, the isoperimetric profile of the square $S_C=[-2,2]\times[-2,2]\subset\R^2$ is
 \begin{equation}\label{expl isop cube} \mathcal{I}_{S_C}(s)= \begin{cases}
     (\pi s)^\frac{1}{2}, & \text{ if  $0 \le s \le \frac {16} \pi $},\\
     4, & \text{ if } \frac {16} \pi  \le s \le 8.
 \end{cases}\end{equation}

\begin{theo}\label{head thm}
Let $\Omega$ be as in \eqref{Omegadef} satisfying \eqref{Lell}.  Any minimizer $D$ of $h_{\mathcal N}(\Omega)$ is such that $|D|<|\Omega|/2$. Moreover, if $l> 1$, then the only minimizer  of $h_{\mathcal N}(\Omega)$ is $E_2$. If $l=1$, then the only minimizers  of $h_{\mathcal N}(\Omega)$ are $E_1,$ $E_2$, and $E_1 \cup E_2$. 
\end{theo}
\begin{proof}
Recall that $\eps:=\frac{1}{44}$ and $\ell\in[1,2].$ Let $D$ be a minimizer of $h_{\mathcal N}(\Omega)$. Then,
\begin{align}\label{hypD}
h_{\mathcal N}(\Omega) = \frac{P(D; \Omega)}{|D|}\qquad \text{ and }\qquad |D|\leq \frac{|\Omega|}{2}
\leq 9+\frac{\ell^2}{2}\leq 11.
\end{align}

\medskip

\underline{\textit{Step 1.} (Comparison and its immediate consequences)}  Let $A_1:=\Omega\cap\{x_1<-2\}$, namely, the square and the bridge on the left. Then $A_1$ is a competitor for $h_{\mathcal N}(\Omega)$ (note that 
$|A_1|=1+ \varepsilon < \frac{|\Omega|}{2}$). Thus,
\begin{align}\label{heps}
h_{\mathcal N}(\Omega) \le \frac{P(A_1; \Omega)}{|A_1|}= \frac{ \varepsilon}{1+\varepsilon} < \varepsilon.    
\end{align}

Hence, by \eqref{hypD},
\begin{align}
\label{e:perimeter bound}
P(D;\Omega) & \le 11 \varepsilon=\frac{1}{4},\qquad \text{because $\varepsilon=\frac{1}{44}$}.
\end{align}  

Next we claim that
\begin{align}\label{claim1}
|D|<\frac{|\Omega|}{2}.    
\end{align}
This holds trivially if $|D \cap S_C| = \emptyset$, since
\begin{align*}
    \frac{1}{2}|\Omega|> 8>1+\ell^2 + 2\varepsilon \ge |D\cap (\Omega\backslash S_C)|=|D|.
\end{align*}
Thus, it only remains to consider the case $|D \cap S_C| > 0$.  We claim that, under this assumption, $P(D \cap S_C; S_C) > 0$. Indeed, if not, $D \cap S_C=S_C$, and
    \[ |D| \ge |D\cap S_C| =|S_C|=16>\frac{1}{2}|\Omega|,\]
which would contradict \eqref{hypD}.  Therefore, $P(D \cap S_C; S_C) \ne 0$. 
By \eqref{e:perimeter bound},
\begin{align*}
    \mathcal{I}_{S_C}(|D\cap S_C|) \le P(D \cap S_C; S_C) \le \frac{1}{4}.
\end{align*}
By \eqref{expl isop cube} this implies that $(\pi |D\cap S_C|)^\frac{1}{2}=I_{S_C}(|D\cap S_C|)\leq \frac{1}{4}$, namely 
\begin{align}\label{s}
|D\cap S_C| \le \frac{1}{16\pi}.    
\end{align}
Hence,
\begin{align*}
|D|
=|D \cap S_C| + |D \cap (\Omega\backslash S_C)|
\le \frac{1}{16\pi} + (2 \varepsilon + 1+ \ell^2) < \frac{|\Omega|}{2},
\end{align*}
and \eqref{claim1} follows.

\medskip

\underline{\textit{Step 2.} (Regularity and case analysis)}
By Theorem~\ref{regn2}, 
\begin{align}\label{e:CMC boundary}
    \text{any portion of the boundary $\partial D$ is part of a circle with radius $r:=\frac{1}{\lambda_{1,1}}\ge \frac{1}{\varepsilon}=44$}
\end{align} 
and
\begin{align}\label{ortho}
\text{the boundary $\partial D$  meets orthogonally $\partial \Omega$ in the points of $\partial \Omega$ which are smooth.}    
\end{align}

First, we demonstrate that every component of $\partial D$ must intersect $\partial \Omega$. Indeed, by \eqref{e:CMC boundary}, any component of $\partial D$ which did not intersect $\partial \Omega$ must be a circle with radius $r\geq 44$ completely contained in $\Omega,$ but this is impossible. Hence, all components of $\partial D $ must intersect $\partial \Omega$.

Now, let $\Gamma \subset \partial D$ be a connected component with endpoints $a, b \in \partial \Omega$. By \eqref{e:perimeter bound}, 
\begin{align}\label{length}
\text{the length of $\Gamma$ is less than $1/4$.  In particular, $|a-b|\leq \frac{1}{4}$.}
\end{align}

\begin{figure}[h!]
    \centering
    \begin{subfigure}{0.108\linewidth}
        \centering
        \includegraphics[width=\linewidth]{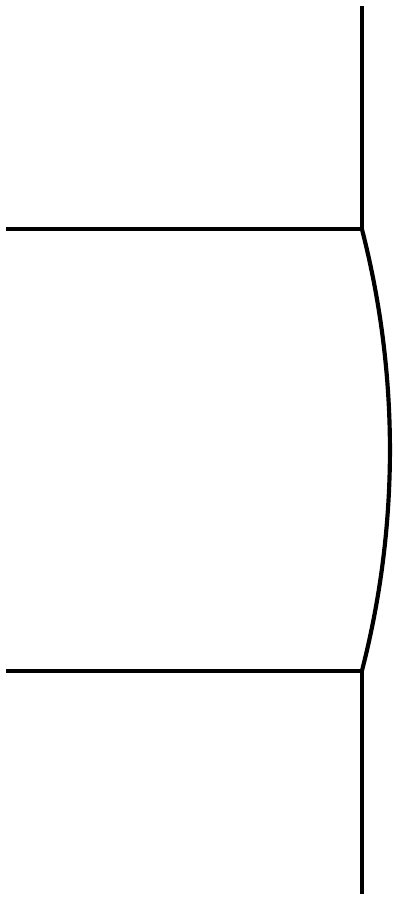}
        \caption{}
    \end{subfigure}
    \hspace{2cm}
    \begin{subfigure}{0.10\linewidth}
        \centering
        \includegraphics[width=\linewidth]{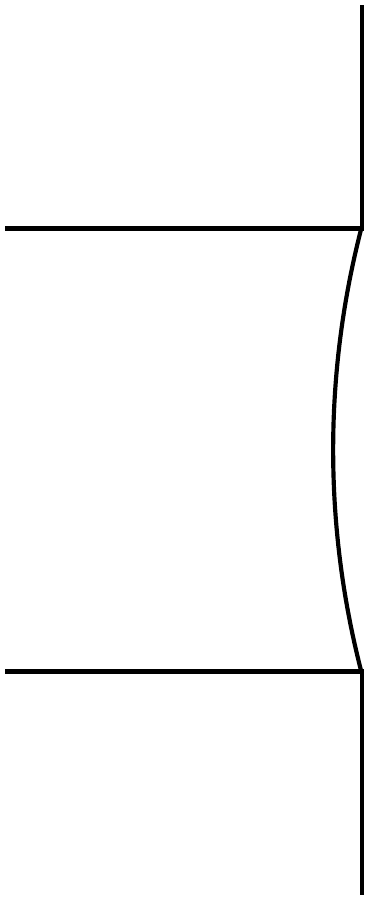}
        \caption{}
    \end{subfigure}
    \hspace{2cm}
    \begin{subfigure}{0.10\linewidth}
        \centering
        \includegraphics[width=\linewidth]{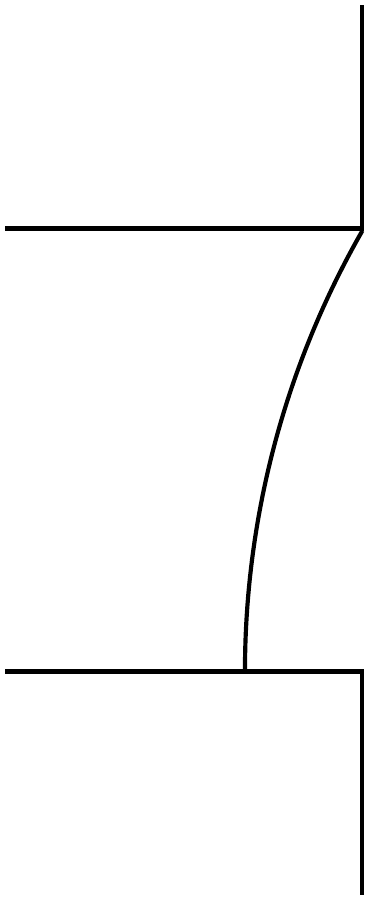}
        \caption{}
    \end{subfigure}
    
    \caption{Possible cuts.}
    \label{cuts}
\end{figure}

Let 
\begin{align*}
T_-:=\partial Q(1, \varepsilon, (-(\tfrac{1}{2}+2), 0))\quad 
\text{ and }\quad 
T_+:=\partial Q(1, \varepsilon, (\tfrac{1}{2}+2, 0))
\end{align*}
 denote the boundaries of the tubes (or bridges) on the left and on the right respectively, see Figure~\ref{hwe}. Now, it is not hard to see that \eqref{e:CMC boundary}, \eqref{ortho}, and \eqref{length} imply that the only possibilities for the endpoints $a$ and $b$ are that $a$ is located at a corner of $T_\pm$ and either
\begin{enumerate}
    \item[$(i)$] $b$ is a regular point on the opposite face of $T_\pm$ (see (c) in Figure~\ref{cuts}).
    \item[$(ii)$] $b$ is also located at a corner of $T_\pm$ (see (a) and (b) in Figure~\ref{cuts}).
\end{enumerate}

For completeness, we next show that $(i)$ and $(ii)$ are indeed the only possibilities.  By \eqref{e:CMC boundary} we know that $\partial D \cap \Omega$ is an arc of a circle with a radius larger than $44$, and, by \eqref{e:perimeter bound}, we know that its length is less than $1/4.$ Hence, if $\Gamma:[0,T]\rightarrow \mathbb{R}^2$ is a unit-speed parametrization of $\partial D \cap \Omega$, we have that
\begin{align}\label{e:angle bound}
    0<\angle(\Gamma'(0), \Gamma'(T)) \le 2\pi\cdot \frac{\frac{1}{4}}{2\pi \cdot \frac{1}{\varepsilon}} =  \frac{\varepsilon}{4} < \frac{\pi}{4},
\end{align}
where $\angle(u,v)$ denotes the angle between two vectors $u,v\in \mathbb R^2$. Now, we do a case analysis.  By \eqref{ortho}, if $a, b$ are in regular points of the left, center, or right squares ($S_L$, $S_C,$ or $S_R$), then $\Gamma$ must turn at least $\pi/2$, since $|a-b| \le \frac{1}{4}$, by \eqref{length}.  But, this contradicts \eqref{e:angle bound}.  Similarly, if $a, b$ are in regular points of the left and right bridges $T_\pm$ or of a bridge $T_\pm$ and one of its adjacent squares, the change in angle must either be at least $\frac{\pi}{4}$ or equal to zero, contradicting \eqref{e:angle bound} or \eqref{e:CMC boundary}, respectively.  This eliminates all possible cases of $a, b$ both being regular points of $\partial \Omega$.

If $a$ or $b$ is located at a corner of one of the squares $S_L,$ $S_C,$ or $S_R$, then by the perimeter length bound \eqref{e:perimeter bound} the other point must be located at a regular point and we arrive at the same contradiction as in the previous paragraph.  A similar argument eliminates the case that $a$ is located at a corner of one of the bridges $T_\pm$ and $b$ is a regular point of an adjacent face.  This proves the claim and $(i)$ and $(ii)$ are the only possibilities. 

\medskip

Now, since $\Gamma \subset \partial D$ was an arbitrary connected component, we see that, by \eqref{s}, $|D\cap S_C| \le \frac{1}{16\pi}$ and therefore the interior of $D$ cannot contain $(0,0)$, since in this case $|D \cap S_C| \ge \Omega \setminus (E_1 \cup E_2) \ge 16 - 2\varepsilon > \frac{1}{2}|\Omega|$. Furthermore, explicit computations show that, in case $(i)$, the perimeter of the cut (c) in Figure~\ref{cuts} is $r \arcsin{\varepsilon/r}$, with $r=1/\lambda_{1,1}$, whereas the cuts (a) and (b) in Figure~\ref{cuts} have perimeter $r \arcsin{\varepsilon/(2r)}$. In particular, this shows that all possible subsets $E\subset \Omega$ with boundaries as described by $(i)$ and $(ii)$ 
and satisfying that their volume is less than $|\Omega|/2$ are subsets of $E_1, E_2$ or $E_1 \cup E_2$, which reduces the options to just these sets.  As a consequence, comparing the area and the perimeter of each possibility, we conclude that, if $\ell\in (1,2]$, then $h_{\mathcal N}(\Omega)$ has a unique minimizer given by $E_2$, whereas if $\ell =1$, then  $h_{\mathcal N}(\Omega)$ only admits $E_1,$ $E_2,$ and $E_1 \cup E_2$ as minimizers, as claimed.
\end{proof}

Theorem~\ref{head thm} and Theorem~\ref{thm cheeger} (see also the proof of Corollary~\ref{coro convex}) yield the following corollary. 
\begin{coro}\label{head cor}
Let $\Omega$ be as in \eqref{Omegadef} and let $u$ be a minimizer of $\lambda_{1,1}(\Omega)$.
\begin{itemize}
\item[(i)]  If $l \ne 1$, then $u=\alpha \chi_{E_2}$, with $\alpha \ne 0$. 
\item[(ii)] If $l=1$, then $u= \alpha \chi_{E_1} + \beta \chi_{E_2}$, with $\alpha, \beta \in \R$, $\alpha+\beta \ne 0$.
\end{itemize}
\end{coro}

We are ready to prove Theorem~\ref{thm:head intro}.

\begin{proof}[Proof of Theorem~\ref{thm:head intro}]
The claim follows from the case $l\ne 1$ in Theorem~\ref{head thm} together with Corollary~\ref{head cor}.     
\end{proof}

\begin{remark}
\begin{enumerate}
\item As a consequence of Theorem~\ref{head thm} and Corollary~\ref{head cor}, given any minimizer $u$ of $\lambda_{1,1}(\Omega)$, one has $|\{u =0 \}| \ne 0$. 
\item If $l \ne 1$, Theorem~\ref{head thm} shows that any minimizer $u$ of $\lambda_{1,1}(\Omega)$ has only one sign, therefore $\int_\Omega \sgn(u) \,d\mathcal{H}^N \ne0$.
\item If $l=1$, then $u$ can be sign-changing. 
\item If $l=1$, then $u= \alpha \chi_{E_1} + \beta \chi_{E_2}$ with $0 < \alpha < \beta$ (or $\alpha< \beta < 0$) is a minimizer.  This contrasts with the convex case studied in Corollary~\ref{coro convex}. 
\end{enumerate}
\end{remark}

\begin{remark}\label{ggremark}
The link between $h_{\mathcal N}(\Omega)$ and eigenvalues of the Neumann 1-Laplacian was studied in \cite{gg,g01}, using ideas from \cite{f60}. In these works, the first eigenvalue is defined via
\begin{align}\label{rmk1}
\widetilde \lambda_{1,1}(\Omega):=\inf \left\{|Du|(\Omega): \, u \in V_1\right \},\qquad 
V_1:=\left\{
u \in BV(\Omega), \, \norm{u}_1=1, \, \int_\Omega \sgn(u)\,d\mathcal{H}^N=0
\right\}.
\end{align}
Observe that $\lambda_{1,1}=\widetilde \lambda_{1,1}$. Indeed, clearly $\lambda_{1,1}\leq \widetilde \lambda_{1,1}.$ Now, let $E$ be a minimizer for $\lambda_{1,1}$, let $\omega\subset \Omega\backslash E$ be such that $|E|=|\omega|,$ and let 
\begin{align*}
w_n:=\beta_n(\chi_{E}-\frac{1}{n}\chi_{\omega})\qquad \text{ with }\beta_n:=\left(\left\|\chi_{E}-\frac{1}{n}\chi_{\omega}\right\|_1\right)^{-1}=\frac{1}{|E|}+o(1)\quad \text{ as }n\to\infty.
\end{align*}
Then $w_n\in BV(\Omega),$ $\norm{w_n}_1=1$, and $\int_\Omega \sgn(w_n)\,d\mathcal{H}^N=0$. Moreover,
\begin{align*}
    \left|D\left(\frac{\chi_{E}}{|E|}-w_n\right)\right|(\Omega)
    =\left|D\left(\frac{\chi_{E}}{|E|}-\left(\frac{1}{|E|}+o(1)\right)(\chi_{E}-\frac{1}{n}\chi_{\omega})\right)\right|(\Omega)
    =o(1)\quad \text{ as }n\to\infty.    
\end{align*}
Namely, $w_n\in V_1$ converges to $\frac{\chi_{E}}{|E|}$ in $BV(\Omega)$, and thus $\lambda_{1,1}\geq \widetilde \lambda_{1,1},$ showing that $\lambda_{1,1}=\widetilde \lambda_{1,1}$. However, in general, $\widetilde \lambda_{1,1}$ is not attained in $V_1$ (as shown in Corollary \ref{head cor}) and this is a natural limitation in \cite{gg,g01} for the analysis of the associated eigenfunctions. In contrast, minimizers of $\lambda_{1,1}$ are achieved, showing that this is a more robust variational framework.  

On the other hand, in \cite[Proposition 6]{BF25}, it is shown in particular that 
\begin{align*}
\inf_{\substack{
u\in W^{1,1}(\Omega)\\
\int_{\Omega} \operatorname{sign}(u)= 0
}}
\frac{\int_{\Omega} |\nabla u|}
{\int_{\Omega} |u|}
=
\min_{\substack{
u\in BV(\Omega;\mathbb{R}^{+})\\
|\{u>0\}|\leq \frac{|\Omega|}{2}
}}
\frac{|Du|(\Omega)}{\int_{\Omega}|u|}
=\min_{\substack{
E\subseteq \Omega\\
|E|\leq \frac{|\Omega|}{2}
}}
\frac{\operatorname{Per}(E;\Omega)}{|E|}.
\end{align*}
Note that the set of competitors $\{u\in BV(\Omega;\mathbb{R}^{+})\::\:|\{u>0\}|\leq \frac{|\Omega|}{2}\}$ only considers nonnegative functions, and therefore it does not include sign-changing eigenfunctions, which are natural limits of minimizers of $\lambda_{p,q}$ as $(p,q)\to (1,1)$; see, for instance, \cite{SdLSdL,DL26}. Furthermore, \cite{BF25} is focused on finding a lower bound for the inverse of the Poincaré constant in $W^{1,1}(\Omega)$, and therefore it does not make any connection with $p$-Laplacian equations for any $p\geq 1.$
\end{remark}

\begin{remark}
    The behavior of the optimal sets for $h_\mathcal{N}(\Omega)$ when $\Omega$ is the domain in Figure \ref{hwe} is similar to what happens in the barbell-type domain within the Cheeger context (see \cite{KS}). In that case, if the two squares have the same area, then the Cheeger set of the first square, the Cheeger set of the second square, and the union of the two, are all Cheeger sets for the domain. If the two squares have different sizes, then the Cheeger set of the domain is the Cheeger set of the largest square. 
\end{remark}

\section*{Disclosure statement}
The authors report there are no competing interests to declare.

\section*{Declaration of generative AI use}
The authors report generative AI was not used in their research or preparation of this manuscript.

\section*{Acknowledgments}
We thank Hugo Tavares for helpful discussions and suggestions. 

Delia Schiera is partially supported by the Portuguese government through FCT Fundação para a Ciência e a Tecnologia and the Recovery and Resilience Plan (PRR) through projects UID/04459/2025 and UID/PRR/04459/2025 (CAMGSD), 2023.17881.ICDT with DOI  10.54499/2023.17881.ICDT (project SHADE), and 2024.14494.PEX with DOI 10.54499/2024.14494.PEX (project ASSO).
She is also supported by GNAMPA (Gruppo Nazionale per l’Analisi, Probabilità e le loro Applicazioni) – INdAM (Istituto
Nazionale di Alta Matematica), through the project ‘Critical and limiting phenomena in nonlinear elliptic
systems’, CUP E5324001950001, and by FCT 
under the Scientific Employment Stimulus - Individual Call (CEEC Individual), with DOI identifier 10.54499/2020.02540.CEECIND/CP1587/CT0008.

This work was developed while Delia Schiera was visiting the Universidad Nacional Autónoma de México: she gratefully acknowledges the financial support of FCT through the FCT Mobility grant FCT/Mobility/1304796440/2024-25, as well as the warm hospitality she received during her stay. 

A. Saldaña is supported by SECIHTI grant CBF2023-2024-116 (Mexico) and by UNAM-DGAPA-PAPIIT grant IN102925 (Mexico).

\bibliographystyle{plain}
\bibliography{refs}

\vspace{1cm}

\noindent\textbf{Delia Schiera}\\
CAMGSD - Centro de An\'alise Matem\'atica, Geometria e Sistemas Din\^amicos\\
Departamento de Matem\'atica do Instituto Superior T\'ecnico\\
Universidade de Lisboa\\
1049-001 Lisboa, Portugal\\
\texttt{delia.schiera@tecnico.ulisboa.pt}

\bigskip

\noindent\textbf{Sean McCurdy} and \textbf{Alberto Saldaña}\\
Instituto de Matemáticas\\
Universidad Nacional Autónoma de México \\
Circuito Exterior, Ciudad Universitaria\\
04510 Coyoacán, Ciudad de México, Mexico\\
\texttt{sean.mccurdy@im.unam.mx} \\
\texttt{alberto.saldana@im.unam.mx}
\medskip
\end{document}